\documentclass[reqno]{amsart}
\usepackage{url}
\newtheorem{theorem}{Theorem}[section]
\newtheorem{lemma}[theorem]{Lemma}
\newtheorem{corollary}[theorem]{Corollary}
\newtheorem{conjecture}[theorem]{Conjecture}

\theoremstyle{definition}
\newtheorem{definition}[theorem]{Definition}

\theoremstyle{remark}

\numberwithin{equation}{section}
\usepackage{algorithmic}
\usepackage[top=3cm,bottom=2cm,right=2cm,left=2cm]{geometry}

\begin{document}

\title[$2$-Elongated Plane Partitions and Powers of 7]{$2$-Elongated Plane Partitions and Powers of 7: The Localization Method Applied to a Genus 1 Congruence Family}


\author{Koustav Banerjee}
\address{}
\curraddr{}
\email{}
\thanks{}

\author{Nicolas Allen Smoot}
\address{}
\curraddr{}
\email{}
\thanks{}



\date{}

\dedicatory{}

\begin{abstract}
Over the last century, a large variety of infinite congruence families have been discovered and studied, exhibiting a great variety with respect to their difficulty.  Major complicating factors arise from the topology of the associated modular curve: classical techniques are sufficient when the associated curve has cusp count 2 and genus 0.  Recent work has led to new techniques that have proven useful when the associated curve has cusp count greater than 2 and genus 0.  We show here that these techniques may be adapted in the case of positive genus.  In particular, we examine a congruence family over the 2-elongated plane partition diamond counting function $d_2(n)$ by powers of 7, for which the associated modular curve has cusp count 4 and genus 1.  We compare our method with other techniques for proving genus 1 congruence families, and conjecture a second congruence family by powers of 7, which may be amenable to similar techniques. 
\end{abstract}

\maketitle

\section{Introduction}

\subsection{Background}

The first great breakthrough in the study of the arithmetic properties of partitions began with the justifiably celebrated infinite congruence families discovered by Ramanujan in 1918 \cite{Ramanujan} and refined by Watson in 1938 \cite{Watson}:

\begin{theorem}\label{thorigconjecR}
\begin{align*}
\text{If } n,\alpha\in\mathbb{Z}_{\ge 1} \text{ and } 24n\equiv 1\pmod{\ell^{\alpha}},&\text{ then } p(n)\equiv 0\pmod{\ell^{\beta}}\text{, with}
\end{align*}\begin{align*}
\beta := 
\begin{cases}
\alpha & \text{ if } \ell\in\{5,11\},\\
\left\lfloor \frac{\alpha}{2} \right\rfloor + 1 & \text{ if } \ell=7.
\end{cases}
\end{align*}
\end{theorem}

The cases of this theorem for $\ell=5,7$ are comparatively easy to understand.  Ramanujan may have understood the proof of the case of $\ell=5$ as early as 1918 \cite{BO}, and the proof of the case of $\ell=7$, published by Watson in 1938 \cite{Watson}, is similar in form.  The case for $\ell=11$ is far more difficult, and a proof was not found before Atkin's work in 1967 \cite{Atkin}.

In the decades since, an enormous variety of partition functions and generalizations have been shown to exhibit divisibility properties which at least superficially resemble those of Ramanujan's congruences.  Moreover---again as with Ramanujan's congruences---the difficulty of understanding these congruence families varies enormously.

The reason that such congruence families occur at all is not fully understood, nor is the substantial range in the difficulty of different families.  We do know that each congruence family for a given partition function $a(n)$ is generally associated with a sequence of meromorphic functions over a compact Riemann surface.  Indeed, if $a(n)$ is enumerated by an eta quotient, then this surface is a classical modular curve $\mathrm{X}_0(N)$ for some $N\ge 1$. The topology of that curve appears to be critical in understanding why some congruence families are easier to prove than others. 

There are two topological numbers of $\mathrm{X}_0(N)$ which are important to us: the genus and the cusp count.  In the case of Theorem \ref{thorigconjecR} for $\ell=5,7$ the genus is 0.  For $\ell=11$ the genus is 1.  The cusp count is 2 in all three cases.

In recent years the study of congruence families associated with modular curves of a more complex topology have given rise to new techniques, especially those embodied in the localization method.  These techniques have been useful when studying congruence families in which the associated modular curve has genus 0 and cusp count 4.

On his first approach using these methods, Smoot argued \cite{Smoot0} that the techniques of localization would not succeed when applied to a congruence family associated with a curve of nonzero genus.  In the central result of this paper we delightfully acknowledge that we were mistaken.  We have successfully applied localization to the following congruence family for the 2-elongated plane partition diamond function $d_2(n)$, introduced by Andrews and Paule \cite{AndrewsPaule}: 

\begin{theorem}\label{Thm12}
Let $n,\alpha\in\mathbb{Z}_{\ge 1}$ such that $8n\equiv 1\pmod{7^{\alpha}}$.  Then $d_2(n)\equiv 0\pmod{7^{\left\lfloor\alpha/2\right\rfloor}}$.
\end{theorem}  This family was conjectured by Banerjee, and is associated with the modular curve $\mathrm{X}_0(14)$.  This curve has cusp count 4 and genus 1.

The implications for this breakthrough are substantial.  Unlike other methods for proving congruence families on curves of nontrivial topology, the techniques of localization are made to exploit an algebraic structure which arises naturally from the curve itself.  These techniques provide a straightforward procedure for anyone who wishes to prove a given congruence family in an algorithmic fashion.  

To date, some families have been shown to be resistant to such techniques.  We have already attempted to adapt such techniques to a proof of the celebrated Andrews--Sellers congruence family \cite{Paule}---whose associated modular curve has genus 1 and cusp count 6.  Our attempts have thus far failed.  Comparing the difficulty of proving this case to our success in proving Theorem \ref{Thm12}, we speculate that the most important contributor to the overall difficulty of a congruence family is the cusp count---the genus appears to play an insignificant role with respect to the accessibility of an associated congruence family to proof by our methods.

This is not to say that the genus has no effect at all---rather, it appears that the genus significantly complicates the ``bookkeeping" aspect of the problem.  An enormous amount of information must be carefully tracked and recorded.  In particular, we study each case of Theorem \ref{Thm12} with a given modular generating function.  We want to represent each of these functions in terms of some convenient basis functions over which we have a greater deal of control.  This representation necessitates some peculiarities in the algebraic structure of the associated localization rings which we must account for.  These problems are shared to some extent with those of Theorem \ref{thorigconjecR} in the case that $\ell=11$.  

\subsection{$k$-Elongated Plane Partition Diamonds}

In previous papers, the authors have examined a generalization of $p(n)$ in the form of the counting function for $k$-elongated plane partition diamonds, $d_k(n)$.  These counting functions are enumerated by the generating function \cite[equation (7.3)]{AndrewsPaule}
\begin{align*}
D_k(q) := \sum_{n=0}^{\infty}d_k(n)q^n = \prod_{m=1}^{\infty}\frac{(1-q^{2m})^k}{(1-q^m)^{3k+1}}.
\end{align*}  Notice that this is a class of generalizations for the unrestricted partition function, as $d_0(n)=p(n)$.  One appeal to studying these functions is that they are interesting arithmetic objects in and of themselves.

However, what makes them particularly interesting to us is the fact that when congruence families arise, they appear without exception to be associated with modular curves of cusp count 4.

Such families have been proved for $d_2(n)$ over powers of 3 by Smoot \cite{Smoot2}, $d_5(n)$ over powers of 5 by Banerjee and Smoot \cite{Banerjee}, and $d_7(n)$ over powers of 2 by Sellers and Smoot \cite{SellersS}.  In these cases, the associated modular curves ($\mathrm{X}_0(6)$, $\mathrm{X}_0(10)$, and $\mathrm{X}_0(8)$ respectively) have genus 0 and cusp count 4.  The case for $d_7(n)$ in \cite{SellersS} is peculiar, in that it admits a much simpler proof than is expected.  However, our proof of Theorem \ref{Thm12} will resemble the proofs in \cite{Smoot2} and \cite{Banerjee} in form.

As in these cases, we begin by building a sequence of generating functions $\mathcal{L}:=\left(L_{\alpha}\right)_{\alpha\ge 1}$, each with a prefactor designed to make each member of the sequence a modular functions, say $(\Phi_{\alpha})_{\alpha \ge 1}$, over the associated congruence subgroup $\Gamma_0(14)$ (equivalently, a meromorphic function over $\mathrm{X}_0(14)$).
\begin{align}
L_{\alpha} := \Phi_{\alpha}\cdot\sum_{8n\equiv 1\bmod{7^{\alpha}}}d_2(n)q^{\left\lfloor n/7^{\alpha}\right\rfloor+1},\label{lalphadefn}
\end{align} with $q := e^{2\pi i\tau}$, $\tau\in\mathbb{H}$.  We also construct an alternating sequence of linear operators, $U^{(0)}$ and $U^{(1)}$, such that 
\begin{align}
L_{\alpha+1} = U^{(\alpha)}\left( L_{\alpha} \right),\label{ulagtlap1}
\end{align} with $U^{(\alpha)}\in\left\{U^{(0)},U^{(1)}\right\}$ depending on the parity of $\alpha$.

The idea is to represent each $L_{\alpha}$ in terms of some convenient reference functions that arise naturally from the space, and then to study the application of $U^{(\alpha)}$ on these functions.

To do this, we note that because $\mathrm{X}_0(14)$ is compact, any function which is holomorphic along the whole curve must be a constant.  Thus, our non-constant functions must have a pole somewhere.  We consider the space of functions on $\mathrm{X}_0(14)$ which have a pole only at a single point of the curve---say, at the cusp $[0]$ (see the next section for details on the construction of $\mathrm{X}_0(14)$). 

The Weierstraß gap theorem \cite{PRW} (see also \cite[Sect. III.5.3]{FK}) dictates that the space of all such functions has the form
\begin{align}
\mathcal{M}^{0}\left( \mathrm{X}_0(14) \right) = \mathbb{C}[x]\oplus y\mathbb{C}[x],\label{c0fspace}
\end{align} for some functions $x,y$.  In particular, $x$ ought to have order $-2$ at $[0]$, while $y$ has order $-3$ at the same.

Notice from (\ref{lalphadefn}) that our functions $L_{\alpha}$ have positive order at $[\infty]$.  If we were working over a classical modular curve of cusp count 2, $L_{\alpha}$ would not have negative order anywhere besides the cusp $[0]$.  However, we are working with $\mathrm{X}_0(14)$, a curve of cusp count 4; consequently, $L_{\alpha}$ may have poles at two other points besides $[0]$.

Therefore, to kill these other possible poles, we consider multiplying $L_{\alpha}$ by some function $z$ which lives in $\mathcal{M}^{0}\left( \mathrm{X}_0(14) \right)$ and which has positive order at all other poles of $L_{\alpha}$.  It is useful that algorithms exist for which we may construct a suitable eta quotient; for example, the construction of $\mu$ in \cite{Radu}.  One such quotient has the form
\begin{align*}
z := z(\tau) &= \frac{(q^2;q^2)_{\infty}^{7}(q^7;q^7)_{\infty}}{(q;q)_{\infty}^7 (q^{14};q^{14})_{\infty}}.
\end{align*}  One can immediately notice that $z\equiv 1\pmod{7}$.  This is useful in that multiplying $L_{\alpha}$ by $z$ will not affect divisibility (or lack thereof) by 7.  Moreover, $z$ has order $-2$ at $[0]$.  If we define 
\begin{align*}
x := \frac{z-1}{7},
\end{align*} then $x$ satisfies the condition of (\ref{c0fspace}).  Thus,
\begin{align*}
(1+7x)^n L_{\alpha}
\end{align*}  can be made a member of $\mathcal{M}^{0}\left( \mathrm{X}_0(14) \right)$ for an appropriately chosen $n$. There are different options for the function $y$.  One notable option is 
\begin{align*}
y_0 := y_0(\tau) &= \frac{(q^2;q^2)_{\infty}^{4}(q^7;q^7)_{\infty}^8}{(q;q)_{\infty}^8 (q^{14};q^{14})_{\infty}^4}.
\end{align*}  It will be slightly more useful to choose
\begin{align*}
y := \frac{y_0-1}{8},
\end{align*} which follows from $y_0\equiv 1\pmod{8}$.  We should expect that each $L_{\alpha}$ has a useful representation as a member of the localization ring $\mathbb{Z}[x]_{\mathcal{S}}\oplus y\mathbb{Z}[x]_{\mathcal{S}}$, in which
\begin{align*}
\mathcal{S} := \left\{ (1+7x)^n : n\ge 0 \right\}.
\end{align*}  The coefficients in the numerator of each expression should then give us some information regarding the divisibility of $L_{\alpha}$ by powers of 7.

Aside from the comparatively easy problem of finding functions $x,y$ which are suitable for us, this procedure so far is very straightforward.  Various complications will indeed emerge from this approach, but they do so in a manner which allows us to very precisely state where these complications arise.

Let us examine $L_1$, which we define in the following manner
\begin{align*}
L_{1} &= \frac{(q^7;q^7)^{7}_{\infty}}{(q^{14};q^{14})^2_{\infty}}\cdot\sum_{n=0}^{\infty} d_2(7n + 1)q^{n+1}
\end{align*} (see Section 2 for a precise definition of every $L_{\alpha}$).  Notice that $L_1\in\mathbb{Z}[[q]]$.  
Generally, we expect that for some integer sequence $\left(\psi(\alpha)\right)_{\alpha\ge 1}$
\begin{align*}
\frac{(1+7x)^{\psi}}{7^{\left\lfloor \alpha/2 \right\rfloor}}\cdot L_{\alpha}\in\mathbb{Z}[x]\oplus y\mathbb{Z}[x].
\end{align*} 
Now in this case we do not expect to find divisibility by 7, since $\left\lfloor 1/2 \right\rfloor = 0$. In particular, we expect
\begin{align*}
(1+7x)^{\psi}\cdot L_{1}\in\mathbb{Z}[x]\oplus y\mathbb{Z}[x].
\end{align*}  However, a strange complication emerges, as we see in examining $L_1$:
\begin{align}
L_1 =& \frac{1}{(1+7x)^3}\Bigg( \frac{320013737}{7}x + \frac{29164229489}{7} x^2 + \frac{1226655768017}{7} x^3 + 
 4505536916704 x^4\label{L1expressxy}\\ &+ 79044206825472 x^5 + 999877459130368 x^6 + 
 9391378522824704 x^7 + 66411983644131328 x^8\nonumber\\ &+ 
 354409645379944448 x^9 + 1415208166316048384 x^{10} + 
 4140177110624894976 x^{11}\nonumber\\ &+ 8532124891883765760 x^{12} + 
 11539756946659737600 x^{13} + 8913467434661314560 x^{14}\nonumber\\ &+ 
 2773078757450186752 x^{15} - \frac{320013688}{7}y - \frac{28844055074}{7}xy - 
 171156188528 x^2 y - 4337927987008 x^3 y\nonumber\\ &- 74846829673728 x^4 y - 
 928384597776384 x^5 y - 8516830910414848 x^6 y - 
 58508210959679488 x^7 y\nonumber\\ &- 300982634640572416 x^8 y - 
 1145123381897592832 x^9 y - 3131903931035156480 x^{10} y\nonumber\\ &- 
 5830893280174276608 x^{11} y - 6623201496588615680 x^{12} y - 
 3466348446812733440 x^{13} y\Bigg).\nonumber
\end{align}  The presence of rational coefficients---indeed, the presence of rational coefficients containing 7 in the denominators---demands our attention and rectification.  Given that $x,y,$ and indeed $L_1$ all have integer series expansions in $q$, we must also have an integer $q$-series expansion for the combination of coefficients of $x$, $x^2$, $x^3$, $y$, and $xy$.  Examining the value of the numerators of these coefficients modulo 7, we can conclude that
\begin{align}
r_L := \frac{1}{7}\left( x+3x^2+x^3+6y+5xy \right)\in\mathbb{Z}[[q]]\label{rLisint}
\end{align}  In particular,
\begin{align*}
x+3x^2+x^3+6y+5xy\equiv 0\pmod{7}.
\end{align*}  The resolution of this problem lies in the behavior of the coefficients of $x^m$, $x^m y$ modulo 7.  Taking the presence of $r_L$ into account, there is no other complication for the representation of $L_{\alpha}$ in terms of $x$ and $y$.  The function $1+7x$ can annihilate the poles of $L_{\alpha}$ without interfering with divisibility of $L_{\alpha}$ by powers of 7.  We now have enough information to state our main theorem:

\begin{theorem}\label{Main}
Let $L_{\alpha}$ be as in \eqref{lalphadefn}, and
\begin{align}
\psi :=& \psi(\alpha) = \left\lfloor \frac{7^{\alpha+1}}{16} \right\rfloor.\label{psidefn}
\end{align}  There exists an integer sequence $(k_{\alpha})_{\alpha\ge 1}$ such that, for all $\alpha\ge 1$, 
\begin{align*}
\frac{(1+7x)^{\psi}}{7^{\beta}}\cdot L_{\alpha} + k_{\alpha}\cdot r_L \in\mathbb{Z}[x]\oplus y\mathbb{Z}[x].
\end{align*}
\end{theorem}

This theorem was proved by an extension of previous applications of the localization method.  We were surprised by this success, as well as the peculiar complications which emerge.

In the first place, the functions $x,y$ are indeed modular, and may be expressed in terms of eta quotients over $\Gamma_0(14)$. However, because of the constant terms neither function is itself an eta quotient.

A second complication emerges in establishing the recurrence relations necessary to complete an induction argument through arbitrary powers of $x$ and $1+7x$.  The presence of a second variable complicates the otherwise straightforward process of deriving a modular equation for $1+7x$.  This is resolved by an important theorem on the relationship between any two meromorphic functions on a compact Riemann surface subject to certain conditions on their respective orders.

The third and most interesting complication emerged right away as we examined $L_1$---the presence of negative powers of 7 which must be accounted for.  We could resolve it in the very first case easily enough, but the general case is much more difficult.  Indeed, this is the first indication of the complex interactions between coefficients of $L_{\alpha}$ which manifest in what we refer to as the \textit{congruence ideal} associated with $L_{\alpha}$.  We give precise definitions for this below; suffice for now to say that when we apply our associated $U^{(\alpha)}$ operators, certain anomalous terms fail to gain divisibility by 7

In order for such a sequence to converge 7-adically to 0 per the conditions of the congruence family, these anomalous terms must somehow cancel out, and this ``cancellation" process must be taken carefully into account.

This third complication is all the more important, given that it is the very last step that can be checked, and must be so examined before completion of the proof.  Without this critical argument, the immensity of work done on a proof may be rendered worthless.

What makes the result of this paper so important is that it provides evidence that this cancellation, expressed here in terms of congruence ideals $I(\alpha)$, can be properly understood in the case that $\epsilon_{\infty}\left( \mathrm{X}_0(N) \right)=4$, even when the genus is nonzero.  That is, if one is faced with a conjectured infinite congruence family in which the associated modular curve is classical with cusp count 4, then the localization method appears to always work.

To express our confidence, we state the following conjecture which ought to be amenable to the techniques in this paper, for the 3-elongated plane partition diamond function $d_3(n)$:

\begin{conjecture}\label{conj12}
Let $n,\alpha\in\mathbb{Z}_{\ge 1}$ such that $6n\equiv 1\pmod{7^{\alpha}}$.  Then $d_3(n)\equiv 0\pmod{7^{\left\lfloor\alpha/2\right\rfloor}}$.
\end{conjecture}  As in the case of Theorem \ref{Thm12}, the associated modular curve is $\mathrm{X}_0(14)$, with genus 1 and cusp count 4.  The reader is happily invited to test our assertion by applying our techniques to this congruence family.

\subsection{Outline}

The remainder of our article is organized as follows.  

In Section \ref{sectionsetup} we define the key generating functions $L_{\alpha}$ which enumerate the cases of Theorem \ref{Thm12}, and the operators $U^{(\alpha)}$ which allow us to map each $L_{\alpha}$ to $L_{\alpha+1}$.  We also define the modular equations for $z$ and $x$.  In Section \ref{sectionalgebra} we define our localization rings $\mathcal{V}^{(\alpha)}$ of modular functions which contain the even- and odd-indexed functions $L_{\alpha}$.  We use our modular equations to derive general relations which detail how $U^{(\alpha)}$ acts on elements of $\mathcal{V}^{(\alpha)}$.  These relations feature certain auxiliary functions $h^{(\alpha)}_{\beta\gamma}$, for which we prove some useful arithmetic properties.

Sections \ref{section0to1full}-\ref{sectionstable} give us the key steps to the proof of Theorem \ref{Main}.  In Section \ref{section0to1full} we show that $U^{(0)}$ takes elements of $\mathcal{V}^{(0)}$ to $\mathcal{V}^{(1)}$.  We confront the more difficult matter of mapping elements of $\mathcal{V}^{(1)}$ to $\mathcal{V}^{(0)}$ in Section \ref{section1to0part}.  We show that $U^{(1)}$ generally does not send every element from $\mathcal{V}^{(1)}$ to $\mathcal{V}^{(0)}$, and we record the deviant terms, which may be bounded to a finite range of variables.

We show in Section \ref{sectionstable} how the deviant terms may be shown to cancel out through construction and manipulation of the congruence ideals associated with each $L_{\alpha}$, thus finishing the proof of Theorem \ref{Main}.  The implications, questions, and future research directions of our work are discussed in Section \ref{sectionfurther}.

Our proof encompasses a great deal of arithmetic, analytic, and topological considerations, as well as a surprising application of commutative algebra.  It also demands computer algebra.  The environment of function spaces associated with this proof is somewhat hostile.  An enormity of calculations was necessary to complete the proof, and it is unreasonable to expect these to be done by hand.  To this end, we do not place all of our computations in this article.  Instead, we include a well-organized Mathematica supplement, which may be found online at \url{https://drive.google.com/file/d/1_5Tyap1P1SgauxO20lzJhWa2SCKDI_vG/view?usp=sharing}, or through direct contact with either author, and which includes the more tedious calculations.  In addition, we also have a second supplementary notebook, \url{https://www3.risc.jku.at/people/nsmoot/FullInitialCases014.nb}, in which some of the more tedious computations in calculating the initial relations (taking an hour or more) are already completed.

\section{Setup}\label{sectionsetup}

For want of space, we will only give an extremely short discussion of the theory of modular functions and Riemann surfaces.  The reader can consult \cite{Diamond} and \cite{Knopp} for a more comprehensive discussion.

\subsection{A Very Brief Overview of the Theory}
Define the congruence subgroup
\begin{align*}
\Gamma_0(N) = \Bigg\{ \begin{pmatrix}
  a & b \\
  c & d 
 \end{pmatrix}\in \mathrm{SL}(2,\mathbb{Z}) : N|c \Bigg\}\le\mathrm{SL}(2,\mathbb{Z}),
\end{align*} for $N\in\mathbb{Z}_{\ge 1}$. If we define the extended complex upper half plane 
\begin{align*}
\hat{\mathbb{H}} := \mathbb{H}\cup\mathbb{Q}\cup\{\infty\},
\end{align*} then $\Gamma_0(N)$ acts on $\hat{\mathbb{H}}$ such that
\begin{align*}
\left(\begin{pmatrix}
  a & b \\
  c & d 
 \end{pmatrix},\tau\right)&\longrightarrow \frac{a\tau+b}{c\tau+d}.
\end{align*}  We denote this action
\begin{align*}
\gamma\tau := \frac{a\tau+b}{c\tau+d},
\end{align*} for $\tau\in\hat{\mathbb{H}}$, $\gamma=\begin{pmatrix}
  a & b \\
  c & d 
 \end{pmatrix}$.  For a given $\tau$, we define the orbits of this group action as
\begin{align*}
[\tau]_N := \left\{ \gamma\tau: \gamma\in\Gamma_0(N) \right\}.
\end{align*}
\begin{definition}
For $N\in\mathbb{Z}_{\ge 1}$, the classical modular curve of level $N$ is the set of all orbits of $\Gamma_0(N)$ applied to $\hat{\mathbb{H}}$:
\begin{align*}
\mathrm{X}_0(N):=\left\{ [\tau]_N : \tau\in\hat{\mathbb{H}} \right\}.
\end{align*}  For $\tau\in\mathbb{Q}\cup\{\infty\}$, $[\tau]_N$ is a \textit{cusp} of $\mathrm{X}_0(N)$.
\end{definition}  The number of cusps is finite and given by a straightforward formula \cite[Figure 3.3]{Diamond}.  Since we are working almost exclusively with $\mathrm{X}_0(14)$, we will sometimes denote
\begin{align*}
[\tau] := [\tau]_{14}.
\end{align*}  There are 4 inequivalent cusps for this modular curve, which we represent here as $\left\{[\infty], [1/7], [1/2], [0]\right\}.$

The curve $\mathrm{X}_0(N)$ is a compact 1-dimensional complex manifold, i.e., a compact Riemann surface.
\begin{definition}\label{DefnModular}
Let $f:\mathbb{H}\longrightarrow\mathbb{C}$ be holomorphic on $\mathbb{H}$.  Then $f$ is a modular function over $\Gamma_0(N)$ if the following properties are satisfied for every $\gamma=\left(\begin{smallmatrix}
  a & b \\
  c & d 
 \end{smallmatrix}\right)\in\mathrm{SL}(2,\mathbb{Z})$:

\begin{enumerate}
\item If $\gamma\in\Gamma_0(N)$, we have $\displaystyle{f\left( \gamma\tau \right) = f(\tau)}.$
\item We have $$\displaystyle{f\left( \gamma\tau \right) = \sum_{n=n_{\gamma}}^{\infty}\alpha_{\gamma}(n)q^{n\gcd(c^2,N)/ N}},$$  with $n_{\gamma}\in\mathbb{Z}$, and $\alpha_{\gamma}(n_{\gamma})\neq 0$.  If $n_{\gamma}\ge 0$, then $f$ is holomorphic at the cusp $[a/c]_N$.  Otherwise, $f$ has a pole of order $n_{\gamma}$, and principal part
 \begin{align*}
 \sum_{n=n_{\gamma}}^{-1}\alpha_{\gamma}(n)q^{n\gcd(c^2,N)/ N},
 \end{align*} at the cusp $[a/c]_N$.
 \end{enumerate}  We refer to $\mathrm{ord}_{a/c}^{(N)}(f) := n_{\gamma}(f)$ as the order of $f$ at the cusp $[a/c]_N$.
\end{definition}

\begin{definition}
Let $\mathcal{M}\left(\Gamma_0(N)\right)$ be the set of all modular functions over $\Gamma_0(N)$, and $\mathcal{M}^{a/c}\left(\Gamma_0(N)\right)\subset \mathcal{M}\left(\Gamma_0(N)\right)$ to be those modular functions over $\Gamma_0(N)$ with a pole only at the cusp $[a/c]_N$.  These are commutative algebras with scalar multiplication by $\mathbb{C}$, and standard addition and multiplication \cite[Section 2.1]{Radu}.
\end{definition}

As might be expected from the way modular functions are defined, there is a one-to-one correspondence between $\mathcal{M}\left(\Gamma_0(N)\right)$ and the set of meromorphic functions on $\mathrm{X}_0(N)$ which are holomorphic everywhere except for the cusps \cite[Chapter VI, Theorem 4A]{Lehner}.  Moreover, the notions of poles, zeros, and orders match between the different sets.

The correspondence itself is the natural
\begin{align*}
f\longrightarrow\hat{f}:\mathrm{X}_0(N)&\longrightarrow\mathbb{C}\cup\{\infty\}\\
:[\tau]_N&\longrightarrow f(\tau).
\end{align*}  
\begin{theorem}\label{riemannsurfacetheorema}
Let $\mathrm{X}$ be a compact Riemann surface, and let $\hat{f}:\mathrm{X}\longrightarrow\mathbb{C}$ analytic on the entirety of $\mathrm{X}$.  Then $\hat{f}$ is be a constant function.
\end{theorem}

\begin{corollary}
For any $N\in\mathbb{Z}_{\ge 1}$, if $f\in\mathcal{M}\left(\Gamma_0(N)\right)$ has nonnegative order at every cusp of $\Gamma_0(N)$, then $f$ is a constant function. 
\end{corollary}

\subsection{Constructing $L_{\alpha}$}

We want to construct the sequence of functions $\mathcal{L}=\left(L_{\alpha}\right)_{\alpha\ge 1}$, each of which has the form (\ref{lalphadefn}), together with a suitable sequence of modified Atkin $U_7$ operators $U^{(\alpha)}$ such that (\ref{ulagtlap1}) is satisfied.  To this end, we construct
\begin{align*}
\mathcal{A} &:= q^6\frac{D_2(q)}{D_2(q^{49})} = q^6\frac{(q^2;q^2)^2_{\infty}(q^{49};q^{49})^{7}_{\infty}}{(q;q)^{7}_{\infty}(q^{98};q^{98})^2_{\infty}}.
\end{align*}  This is a modular function over $\Gamma_0(98)$, which we have checked in the supplementary Mathematica file, mentioned before. 

A useful bookeeping notation for us will be to denote
\begin{align*}
(\alpha)\in\{0,1\}:\ (\alpha)\equiv \alpha\pmod{2},
\end{align*} for any $\alpha\in\mathbb{Z}$.  With this in mind, we define the sequence of linear operators $\left( U^{(\alpha)} \right)_{\alpha\ge 0}$ by
\begin{align*}
U^{(1)}(f) :=& U_7(f),\\
U^{(0)}(f) :=& U_7(\mathcal{A}\cdot f),
\end{align*} where $U_7$ is the following classical operator:
\begin{definition}
Let $f(q) = \sum_{m\ge M}a(m)q^m$.  Then
\begin{align*}
U_{\ell}\left(f(q)\right) := \sum_{\ell m\ge M} a(\ell m)q^m.
\end{align*}
\end{definition}
Many of the key properties of $U_{\ell}$ are listed in \cite[Chapter 8]{Knopp}.  We note in particular that $U_{\ell}$ is linear, and that for any two integer power series $f,g$ in $q=e^{2\pi i\tau}$, we have
\begin{align}
U_{\ell}\left(f(q^{\ell})g(q)\right) = f(q) U_{\ell}\left(g(q)\right).\label{ufgreduceA}
\end{align}
We can now define our main function sequence $\left( L_{\alpha} \right)_{\alpha\ge 0}$  
by
\begin{align*}
L_{2\alpha-1}(\tau) &= \frac{(q^7;q^7)^{7}_{\infty}}{(q^{14};q^{14})^2_{\infty}}\cdot\sum_{n=0}^{\infty} d_2(7^{2\alpha-1}n + \lambda_{2\alpha-1})q^{n+1} = \frac{1}{D_2(q^7)}\cdot\sum_{n=0}^{\infty} d_2(7^{2\alpha-1}n + \lambda_{2\alpha-1})q^{n+1},\\
L_{2\alpha}(\tau) &= \frac{(q;q)^{7}_{\infty}}{(q^2;q^2)^2_{\infty}}\cdot\sum_{n=0}^{\infty} d_2(7^{2\alpha}n + \lambda_{2\alpha})q^{n+1} =\frac{1}{D_2(q)}\cdot\sum_{n=0}^{\infty} d_2(7^{2\alpha}n + \lambda_{2\alpha})q^{n+1}.
\end{align*}  Here, each $\lambda_{\alpha}$ represents the minimal positive inverse of 8 modulo $7^{\alpha}$.  In particular, we compute that
\begin{align*}
\lambda_{2\alpha-1} &:= \frac{1+7^{2\alpha-1}}{8},\\
\lambda_{2\alpha} &:= \frac{1+7^{2\alpha+1}}{8} = \lambda_{2\alpha+1}.
\end{align*}

\begin{lemma}
For all $\alpha\ge 1$, we have
\begin{align*}
L_{\alpha+1} &= U^{(\alpha)}\left( L_{\alpha} \right).
\end{align*}
\end{lemma}

\begin{proof}

Using the standard properties of $U_{\ell}$ for any prime $\ell$, we have the following for $\alpha\ge 1$:
\begin{align*}
U^{(2\alpha-1)}\left( L_{2\alpha-1} \right) &=U_7\left( L_{2\alpha-1} \right)\\ &= U_7\left( \frac{1}{D_2(q^7)} \sum_{n\ge 0} d_2\left(7^{2\alpha-1}n + \lambda_{2\alpha-1}\right)q^{n+1} \right)\\
&= \frac{1}{D_2(q)}\cdot U_7\left( \sum_{n\ge 1}d_2\left(7^{2\alpha-1}(n-1) + \lambda_{2\alpha-1}\right)q^{n} \right)\\
&= \frac{1}{D_2(q)}\cdot \sum_{7n\ge 1}d_2\left(7^{2\alpha-1}(7n-1) + \lambda_{2\alpha-1}\right)q^{n}\\
&= \frac{1}{D_2(q)}\cdot \sum_{n\ge 1}d_2\left(7^{2\alpha}n-7^{2\alpha-1} + \lambda_{2\alpha-1}\right)q^{n}\\
&= \frac{1}{D_2(q)}\cdot \sum_{n\ge 0}d_2\left(7^{2\alpha}n+7^{2\alpha}-7^{2\alpha-1} + \frac{1+7^{2\alpha-1}}{8} \right)q^{n+1}\\
&= \frac{1}{D_2(q)}\cdot \sum_{n\ge 0}d_2\left(7^{2\alpha}n+\lambda_{2\alpha}\right)q^{n+1}.
\end{align*}
\begin{align*}
U^{(2\alpha)}\left( L_{2\alpha} \right) &=U_7\left( \mathcal{A}\cdot L_{2\alpha} \right)\\ &= U_7\left( q^6\frac{D_2(q)}{D_2(q^{49})}\frac{1}{D_2(q)} \sum_{n\ge 0} d_2\left(7^{2\alpha}n + \lambda_{2\alpha}\right)q^{n+1} \right)\\
&= \frac{1}{D_2(q^7)}\cdot U_7\left( \sum_{n\ge 7}d_2\left(7^{2\alpha}(n-7) + \lambda_{2\alpha}\right)q^{n} \right)\\
&= \frac{1}{D_2(q^7)}\cdot \sum_{7n\ge 7}d_2\left(7^{2\alpha}(7n-7) + \lambda_{2\alpha}\right)q^{n}\\
&= \frac{1}{D_2(q^7)}\cdot \sum_{n\ge 1}d_2\left(7^{2\alpha+1}n-7^{2\alpha+1} + \lambda_{2\alpha}\right)q^{n}\\
&= \frac{1}{D_2(q^7)}\cdot \sum_{n\ge 0}d_2\left(7^{2\alpha+1}(n+1)-7^{2\alpha+1} + \lambda_{2\alpha}\right)q^{n+1}\\
&= \frac{1}{D_2(q^7)}\cdot \sum_{n\ge 0}d_2\left(7^{2\alpha+1}n+\lambda_{2\alpha+1}\right)q^{n+1}.
\end{align*}
\end{proof}

We note also that 
\begin{align*}
U^{(2\alpha)}\left( 1 \right) &=U_7\left( \mathcal{A}\right) = U_7\left( q^6\frac{D_2(q)}{D_2(q^{49})}\right) = \frac{1}{D_2(q^7)}\cdot U_7\left( q^6 D_2(q)\right)\\
&= \frac{1}{D_2(q^7)}\cdot U_7\left( \sum_{n=0}^{\infty}d_k(n)q^{n+6} \right) = \frac{1}{D_2(q^7)}\cdot U_7\left( \sum_{n\ge 6} d_k(n-6)q^{n} \right)\\
&= \frac{1}{D_2(q^7)}\cdot \sum_{7n\ge 6} d_k(7n-6)q^{n} = \frac{1}{D_2(q^7)}\cdot \sum_{n\ge 0}d_2\left(7n+1\right)q^{n+1}\\
U^{(2\alpha)}\left( 1 \right) &= L_1.
\end{align*}  This will be useful in proving (\ref{L1expressxy}).

Now that we have a well-defined function sequence, as well as a means of going from one member of the sequence to its successor, we now need to consider the functions $L_{\alpha}$ in terms of some useful reference functions.  We defined $z, x, y$ in the introduction.  We now build a modular equation for $z$ that will allow us to build recurrence relations to describe how $U^{(\alpha)}$ affects rational polynomials in $x, y$:

\begin{theorem}\label{modeqnz}
For $0\le k\le 14$, let $b_k\in\mathbb{Z}[X]$ be defined as in the Appendix.  We have
\begin{align}
z^{14}+\sum_{k=0}^{13}b_k(z(7\tau))z^k=0.\label{actmodeqnz}
\end{align}
\end{theorem}  Notice that by (\ref{ufgreduceA}), $U_{7}\left( b_k(z(7\tau)) \right)=b_k(z(\tau))$.

By substituting $z=1+7x$ into (\ref{actmodeqnz}) and simplifying, we immediately have the following:

\begin{corollary}\label{cor1}
For $0\le j\le 13$, let $a_j\in\mathbb{Z}[X]$ be defined as in the Appendix.  We have
\begin{align}
x^{14}+\sum_{j=0}^{13}a_j(x(7\tau))x^j=0.\label{modeqnx}
\end{align}
\end{corollary}

We prove Theorem \ref{modeqnz} in our Mathematica supplement.

\section{Algebra Structure}\label{sectionalgebra}

\subsection{Function Spaces}

We may express the space of modular functions over $\mathrm{X}_0(14)$ with a pole only at the cusp $[0]$ as $\mathbb{C}[x]\oplus y\mathbb{C}[x]$.  We have seen from the form of $L_1$ that we need to work over a broader space.  To do this, we define the following two sets (remembering that $(\alpha)\in\{0,1\}$):  
\begin{align*}
\mathcal{V}^{(\alpha)}_{n} := \left\{ \frac{1}{(1+7x)^n}\sum_{\substack{0\le\beta\le 1\\ m\ge 1-\beta}} s_{\beta}(m)\cdot 7^{\theta^{(\alpha)}_{\beta}(m)}\cdot y^{\beta}x^m: \text{ $s_{\beta}$ is $\mathbb{Z}$-valued and has finite support} \right\}.
\end{align*}  We have a total of four functions $\theta^{(\alpha)}_{\beta}$ which we define in the Appendix.  At times we will also refer to the more general spaces
\begin{align*}
\mathcal{V}^{(\alpha)} := \bigcup_{n\ge 1} \mathcal{V}^{(\alpha)}_{n}.
\end{align*}

Having defined both the operators $U^{(\alpha)}$ which take $L_{\alpha}$ to $L_{\alpha+1}$, as well as the parent spaces $\mathcal{V}^{(\alpha)}$ of $L_{\alpha}$, we now need to study how the former acts on the latter.  We have the following general relation:

\begin{theorem}\label{thmrelAA}

Let $\beta\in\{0,1\}$, $m,n\in\mathbb{Z}$ such that $n\ge 1$, $m\ge 1-\beta$.  Then there exists arrays $h^{(\alpha)}_{\beta\gamma}(m,n,r)\in \mathbb{Z}$ which have finite support in $r$, such that
\begin{align*}
U^{(\alpha)}\left( \frac{y^{\beta}x^{m}}{(1+7x)^{n}} \right) =& \frac{1}{(1+7x)^{7n+\kappa}}\sum_{\substack{0\le\gamma\le 1\\ r\ge 1-\gamma}} h^{(\alpha)}_{\beta\gamma}(m,n,r) 7^{\pi^{(\alpha)}_{\beta\gamma}(m,r)}y^{\gamma}x^r.
\end{align*}  In particular, the functions $\pi^{(\alpha)}_{\beta\gamma}$ are defined in the Appendix, and
\begin{align*}
\kappa = \begin{cases}
0, & (\alpha)=1,\\
3, & (\alpha)=0.
\end{cases}
\end{align*}
\end{theorem}

We show that this general relation may be proved by explicit computation for $1\le m\le 14$, $1\le n\le 14$, beyond which the relation may be proved via recurrence properties incurred by our modular equations.

\begin{lemma}\label{lowerboundpiabc}
Suppose that we define 
\begin{align}
\hat{\pi}^{(\alpha)}_{\beta\gamma}(m,r) := \left\lfloor \frac{7r-m+\epsilon^{(\alpha)}_{\beta\gamma}}{9} \right\rfloor\label{pihatvalues}
\end{align} for some $\epsilon^{(\alpha)}_{\beta\gamma}\in\mathbb{Z}$.  If the relation
\begin{align}
U^{(\alpha)}\left( \frac{y^{\beta}x^{m}}{(1+7x)^{n}} \right) =& \frac{1}{(1+7x)^{7n+\kappa}}\sum_{\substack{0\le\gamma\le 1\\ r\ge 1-\gamma}} h^{(\alpha)}_{\beta\gamma}(m,n,r) 7^{\hat{\pi}^{(\alpha)}_{\beta\gamma}(m,r)}y^{\gamma}x^r\label{lowerbounduaymoz}
\end{align} holds for $1\le m\le 14$, $1\le n\le 14$, then it must also hold for all $m,n\ge 15$.
\end{lemma}

\begin{proof}

We begin by taking (\ref{modeqnz}):
\begin{align*}
z^{14}+\sum_{k=0}^{13}b_k(z(7\tau))z^k= \sum_{k=1}^{14}b_k(z(7\tau))z^k + b_0(z(7\tau)).
\end{align*}  Rearranging and dividing through by $b_0(z(7\tau))$, we have
\begin{align*}
1 =& \frac{-1}{b_0(z(7\tau))}\sum_{k=1}^{14}b_k(z(7\tau))\cdot z^{k}.
\end{align*}  We now divide both sides by $z^n$:
\begin{align*}
\frac{1}{z^{n}} =& \frac{-1}{z(7\tau)^{14}}\sum_{k=1}^{14}b_k(z(7\tau))\cdot \frac{1}{z^{n-k}}.
\end{align*}  We now multiply by a given power of $x$ and $y$:
\begin{align*}
\frac{y^{\beta}x^m}{z^{n}} =& \frac{-1}{z(7\tau)^{14}}\sum_{k=1}^{14}b_k(z(7\tau))\frac{y^{\beta}x^m}{z^{n-k}}.
\end{align*}  Applying $U^{(\alpha)}$ to both sides and and substituting $z=1+7x$, we have
\begin{align}
U^{(\alpha)}\left(\frac{y^{\beta}x^m}{(1+7x)^{n}}\right) =& \frac{-1}{(1+7x)^{14}}\sum_{k=1}^{14}b_k(z(\tau)) U^{(\alpha)}\left(\frac{y^{\beta}x^m}{(1+7x)^{n-k}}\right).\label{ualphaxmznmk}
\end{align}  Next, we recall (\ref{modeqnx}):
\begin{align*}
x^{14}+\sum_{j=0}^{13}a_j(x(7\tau))x^j=0.
\end{align*} 
If we multiply both sides by $\frac{y^{\beta}x^{m-14}}{z^{n-k}}$ for some $m\ge 14$, we have
\begin{align*}
\frac{y^{\beta}x^{m}}{z^{n-k}}+\sum_{j=0}^{13}a_j(x(7\tau))\frac{y^{\beta}x^{m+j-14}}{z^{n-k}}=0.
\end{align*}  Applying $U^{(\alpha)}$, rearranging, and again substituting $z=1+7x$, we have
\begin{align*}
U^{(\alpha)}\left(\frac{y^{\beta}x^{m}}{(1+7x)^{n-k}}\right) = -\sum_{j=0}^{13}a_j(x(\tau)) U^{(\alpha)}\left(\frac{y^{\beta}x^{m+j-14}}{(1+7x)^{n-k}}\right).
\end{align*}  Substituting this into (\ref{ualphaxmznmk}), we now also take $n\ge 14$.  We show that if (\ref{lowerbounduaymoz}) applies to $U^{(\alpha)}\left( \frac{y^{\beta}x^{m+j-14}}{(1+7x)^{n-k}} \right)$ for $0\le j\le 13$, $1\le k\le 14$, then the relation must apply to $U^{(\alpha)}\left( \frac{y^{\beta}x^{m}}{(1+7x)^{n}} \right)$.  Thus, if we verify the relation for the first 14 consecutive integral values of $m,n$, then the relation will apply for all higher $m,n$.  To this end, we have
\begin{align*}
U^{(\alpha)}&\left( \frac{y^{\beta}x^{m}}{(1+7x)^{n}} \right) = \frac{-1}{(1+7x)^{14}}\sum_{k=1}^{14}b_k(z(\tau))\cdot U^{(\alpha)}\left( \frac{y^{\beta}x^{m}}{(1+7x)^{n-k}} \right)\\
=& \frac{1}{(1+7x)^{14}}\sum_{j=0}^{13}\sum_{k=1}^{14}a_j(\tau)b_k(z(\tau))\cdot U^{(\alpha)}\left( \frac{y^{\beta}x^{m+j-14}}{(1+7x)^{n-k}} \right)\\
=& \frac{1}{(1+7x)^{14}}\sum_{j=0}^{13}\sum_{k=1}^{14}a_j(\tau)b_k(z(\tau))\cdot\frac{1}{(1+7x)^{7(n-k)+\kappa}}\sum_{\substack{0\le\gamma\le 1,\\ r\ge 1-\gamma}} h^{(\alpha)}_{\beta\gamma}(m+j-14,n-k,r) 7^{\pi^{(\alpha)}_{\beta\gamma}(m+j-14,r)}y^{\gamma}x^r\\
=& \frac{1}{(1+7x)^{7n+\kappa}}\sum_{j=0}^{13}\sum_{k=1}^{14}w(j,k) \sum_{\substack{0\le\gamma\le 1,\\ r\ge 1-\gamma}} h^{(\alpha)}_{\beta\gamma}(m+j-14,n-k,r) 7^{\pi^{(\alpha)}_{\beta\gamma}(m+j-14,r)}y^{\gamma}x^r,
\end{align*} where we define
\begin{align*}
w(j,k) :=& a_j(\tau)b_k(z(\tau))(1+7x)^{7(k-2)}\\
=& \sum_{l=1}^{L} v(j,k,l)\cdot 7^{\left\lfloor \frac{7l+j-6}{9} \right\rfloor}\cdot x^l.
\end{align*}  We verify that $w(j,k)$ has this form in our Mathematica supplement.  This gives us
\begin{align*}
U^{(\alpha)}&\left( \frac{y^{\beta}x^{m}}{(1+7x)^{n}} \right)\\
=& \frac{1}{(1+7x)^{7n+\kappa}}\sum_{\substack{0\le j\le 13,\\ 1\le k\le 14\\ 0\le l\le L\\0\le\gamma\le 1\\ r\ge 1-\gamma}}v(j,k,l) h^{(\alpha)}_{\beta\gamma}(m+j-14,n-k,r) 7^{\pi^{(\alpha)}_{\beta\gamma}(m+j-14,r)+\left\lfloor \frac{7l+j-6}{9} \right\rfloor}y^{\gamma}x^{r+l}\\
=& \frac{1}{(1+7x)^{7n+\kappa}}\sum_{\substack{0\le j\le 13,\\ 1\le k\le 14\\ 0\le l\le L\\0\le\gamma\le 1\\ r\ge 1+l-\gamma}}v(j,k,l) h^{(\alpha)}_{\beta\gamma}(m+j-14,n-k,r-l) 7^{\pi^{(\alpha)}_{\beta\gamma}(m+j-14,r-l)+\left\lfloor \frac{7l+j-6}{9} \right\rfloor}y^{\gamma}x^{r},
\end{align*} with the latter line coming from adjusting $r$.  Examining the powers of 7, we note that 
\begin{align*}
\pi^{(\alpha)}_{\beta\gamma}(m+j-14,r)+\left\lfloor \frac{7l+j-6}{9} \right\rfloor &= \left\lfloor \frac{7r-(m+j-14)+\epsilon^{(\alpha)}_{\beta\gamma}}{9} \right\rfloor +\left\lfloor \frac{7l+j-6}{9} \right\rfloor\\
&\ge \left\lfloor \frac{7(r+l)-m+\epsilon^{(\alpha)}_{\beta\gamma}}{9} \right\rfloor\\
&= \pi^{(\alpha)}_{\beta\gamma}(m,r+l).
\end{align*}  That is, for $r\ge l$,
\begin{align*}
\pi^{(\alpha)}_{\beta\gamma}(m+j-14,r-l)+\left\lfloor \frac{7l+j-6}{9} \right\rfloor &\ge \pi^{(\alpha)}_{\beta\gamma}(m,r).
\end{align*}  Therefore, the power of 7 necessary for (\ref{lowerbounduaymoz}) is achieved in the coefficient of $y^{\gamma}x^r$.  We now note that we can define
\begin{align*}
h^{(\alpha)}_{\beta\gamma}(m,n,r):=\sum_{\substack{0\le j\le 13,\\ 1\le k\le 14\\ 0\le l\le L}}v(j,k,l) h^{(\alpha)}_{\beta\gamma}(m+j-14,n-k,r-l) 7^{\pi^{(\alpha)}_{\beta\gamma}(m+j-14,r-l)+\left\lfloor \frac{7l+j-6}{9} \right\rfloor - \pi^{(\alpha)}_{\beta\gamma}(m,r)},
\end{align*} since the latter also has finite support in $r$.
\end{proof}

\begin{lemma}\label{lowerboundpiabc0}
Let $\pi^{(\alpha)}_{\beta\gamma}$ be defined as in the Appendix.  If the relation
\begin{align}
U^{(\alpha)}\left( \frac{y^{\beta}x^{m}}{(1+7x)^{n}} \right) =& \frac{1}{(1+7x)^{7n+\kappa}}\sum_{\substack{0\le\gamma\le 1\\ r\ge 1-\gamma}} h^{(\alpha)}_{\beta\gamma}(m,n,r) 7^{\pi^{(\alpha)}_{\beta\gamma}(m,r)}y^{\gamma}x^r\label{adsfadaiie33}
\end{align} holds for a fixed $m$ and $1\le n\le 14$, then it must hold for all $n\ge 15$.
\end{lemma}

\begin{proof}

We suppose that (\ref{adsfadaiie33}) holds for $U^{(\alpha)}\left( \frac{y^{\beta}x^{m}}{(1+7x)^{n-k}} \right)$ for some fixed $m\ge 1$ and $n$ with $1\le k\le 14$.  We show that the relation will then apply to $U^{(\alpha)}\left( \frac{y^{\beta}x^{m}}{(1+7x)^{n}} \right)$.

Returning to (\ref{ualphaxmznmk}), we have
\begin{align*}
U^{(\alpha)}&\left( \frac{y^{\beta}x^{m}}{(1+7x)^{n}} \right)\nonumber\\ =& \frac{-1}{(1+7x)^{14}}\sum_{k=1}^{14}b_k(z(\tau))\cdot U^{(\alpha)}\left( \frac{y^{\beta}x^{m}}{(1+7x)^{n-k}} \right)\\
=& \frac{-1}{(1+7x)^{14}}\sum_{k=1}^{14} \frac{b_k(z(\tau))}{(1+7x)^{7(n-k)+\kappa}} \sum_{\substack{0\le\gamma\le 1,\\ r\ge 1-\gamma}} h^{(\alpha)}_{\beta,\gamma}(m,n-k,r)\cdot 7^{\pi^{(\alpha)}_{\beta,\gamma}(m,r)}\cdot y^{\gamma}x^r\\
=&\frac{1}{(1+7x)^{7n+\kappa}}\sum_{k=1}^{14} \hat{w}(k)\sum_{\substack{0\le\gamma\le 1,\\ r\ge 1-\gamma}} h^{(\alpha)}_{\beta,\gamma}(m,n-k,r)\cdot 7^{\pi^{(\alpha)}_{\beta,\gamma}(m,r)}\cdot y^{\gamma}x^r,
\end{align*} in which we can expand
\begin{align*}
\hat{w}(k) :&= -b_k(z(\tau))(1+7x)^{7(k-2)}\\
&=\begin{cases}
& \displaystyle{\sum_{l=0}^{84} \hat{v}(k,l)\cdot 7^{\phi(l)}\cdot x^l},\ k\neq 7,14\\
&\displaystyle{25398809+\sum_{l=1}^{84} \hat{v}(7,l)\cdot 7^{\phi(l)}\cdot x^l},\ k=7,\\
&\displaystyle{-1+\sum_{l=1}^{84} \hat{v}(14,l)\cdot 7^{\phi(l)}\cdot x^l},\ k=14,
\end{cases}\\
\phi(l) :&= \left\lfloor \frac{7l+17}{9} \right\rfloor.
\end{align*}  From the Appendix, we note that $b_0$ has minimum degree 14, and $b_1$ has minimum degree 7.  As such, $\hat{w}(k)$ is always a polynomial in $x$.

Notice also that $25398809\equiv 2\bmod{49}$.  We can now express
\begin{align}
U^{(\alpha)}&\left( \frac{y^{\beta}x^{m}}{(1+7x)^{n}} \right) = \frac{1}{(1+7x)^{7n+\kappa}}\nonumber\\
\times & \Bigg(\sum_{\substack{0\le\gamma\le 1,\\ k\neq 7,14,\\ 0\le l\le 84,\\ r\ge 1-\gamma}} \hat{v}(k,l)\cdot h^{(\alpha)}_{\beta,\gamma}(m,n-k,r)\cdot 7^{\pi^{(\alpha)}_{\beta,\gamma}(m,r) + \phi(l)}\cdot y^{\gamma}x^{r+l}\label{reliancenm5a3}\\
&+\sum_{\substack{0\le\gamma\le 1\\ k=7,14,\\ 1\le l\le 84,\\ r\ge 1-\gamma}} \hat{v}(k,l)\cdot h^{(\alpha)}_{\beta,\gamma}(m,n-k,r)\cdot 7^{\pi^{(\alpha)}_{\beta,\gamma}(m,r) + \phi(l)}\cdot y^{\gamma}x^{r+l}\label{reliancenm5a2}\\
&+\sum_{\substack{0\le\gamma\le 1\\ r\ge 1-\gamma}}\left( 25398809 h^{(\alpha)}_{\beta,\gamma}(m,n-7,r) - h^{(\alpha)}_{\beta,\gamma}(m,n-14,r)\right)\cdot 7^{\pi^{(\alpha)}_{\beta,\gamma}(m,r)} \cdot y^{\gamma}x^{r}\Bigg).\label{reliancenm5a}
\end{align}  Relabeling our powers of $x$, we have
\begin{align}
U^{(\alpha)}&\left( \frac{y^{\beta}x^{m}}{(1+7x)^{n}} \right) = \frac{1}{(1+7x)^{7n+\kappa}}\nonumber\\
\times & \Bigg(\sum_{\substack{0\le\gamma\le 1,\\ k\neq 7,14,\\ 0\le l\le 84,\\ r\ge 1+l-\gamma}} \hat{v}(k,l)\cdot h^{(\alpha)}_{\beta,\gamma}(m,n-k,r-l)\cdot 7^{\pi^{(\alpha)}_{\beta,\gamma}(m,r-l) + \phi(l)}\cdot y^{\gamma}x^{r}\label{reliancenm5a3b}\\
&+\sum_{\substack{0\le\gamma\le 1\\ k=7,14,\\ 1\le l\le 84,\\ r\ge 1+l-\gamma}} \hat{v}(k,l)\cdot h^{(\alpha)}_{\beta,\gamma}(m,n-k,r-l)\cdot 7^{\pi^{(\alpha)}_{\beta,\gamma}(m,r-l) + \phi(l)}\cdot y^{\gamma}x^{r}\label{reliancenm5a2b}\\
&+\sum_{\substack{0\le\gamma\le 1\\ r\ge 1-\gamma}}\left( 25398809 h^{(\alpha)}_{\beta,\gamma}(m,n-7,r) - h^{(\alpha)}_{\beta,\gamma}(m,n-14,r)\right)\cdot 7^{\pi^{(\alpha)}_{\beta,\gamma}(m,r)} \cdot y^{\gamma}x^{r}\Bigg).\label{reliancenm5b}
\end{align}  For any fixed $m$, $\pi^{(\alpha)}_{\beta\gamma}(m,r)$ has the form
\begin{align*}
\pi^{(\alpha)}_{\beta\gamma}(m,r) = \left\lfloor \frac{7r+M}{9} \right\rfloor
\end{align*} for some $M\in\mathbb{Z}$.  With this in mind, we have
\begin{align*}
\pi^{(\alpha)}_{\beta\gamma}(m,r) + \phi(l) =& \left\lfloor \frac{7r+M}{9} \right\rfloor + \left\lfloor \frac{7l+17}{9} \right\rfloor\ge \left\lfloor \frac{7(r+l)+M+9}{9} \right\rfloor\ge \pi^{(\alpha)}_{\beta\gamma}(m,r+l) + 1.
\end{align*}  Relabeling, we have
\begin{align*}
\pi^{(\alpha)}_{\beta\gamma}(m,r-l) + \phi(l) \ge \pi^{(\alpha)}_{\beta\gamma}(m,r) + 1,
\end{align*} and the dominating power of 7 is $\pi^{(\alpha)}_{\beta\gamma}(m,r)$ in (\ref{reliancenm5b}).
\end{proof}

\begin{proof}[Proof of Theorem \ref{thmrelAA}]
Every function $\pi^{(\alpha)}_{\beta\gamma}$ in the Appendix has an associated form $\hat{\pi}^{(\alpha)}_{\beta\gamma}$ (\ref{pihatvalues}) for $m\ge 4$.  For $1\le m\le 3$,
\begin{align*}
\pi^{(\alpha)}_{\beta\gamma}(m,r)\ge\hat{\pi}^{(\alpha)}_{\beta\gamma}(m,r).
\end{align*}  Therefore, we can first use Lemma \ref{lowerboundpiabc} by checking the cases $1\le m\le 14$, $1\le n\le 14$, thus proving Theorem \ref{thmrelAA} in all but some cases for $1\le m\le 3$, and $m=0$.  For $1\le m\le 3$ we verify these cases for $1\le n\le 14$.  By Lemma \ref{lowerboundpiabc0}, we have verified these exceptional cases.

We similarly check the exceptional cases for $m=0$.
\end{proof}

\subsection{Auxiliary Coefficients}

As a consequence of Lemma \ref{lowerboundpiabc0}, we have the following very important result:

\begin{theorem}\label{hcongred37}
The congruence
\begin{align}
h^{(\alpha)}_{\beta\gamma}(m,n,r)\equiv h^{(\alpha)}_{\beta\gamma}(m,n+7,r)\pmod{49}.\label{congcondh}
\end{align} applies in the following cases:
\begin{align}
&h^{(1)}_{00}(m,n,r),\ 1\le m\le 4,\ 1\le r\le 14,\label{congcondh100}\\
&h^{(1)}_{00}(m,n,r),\ m=5,\ r=1,\label{congcondh100a}\\
&h^{(1)}_{01}(m,n,r),\ 1\le m\le 4,\ 1\le r\le 14,\label{congcondh101}\\
&h^{(1)}_{10}(m,n,r),\ 0\le m\le 2,\ 1\le r\le 14,\label{congcondh110}\\
&h^{(1)}_{10}(m,n,r),\ m=3,\ r=1\label{congcondh110a}\\
&h^{(1)}_{11}(m,n,r),\ 1\le m\le 14,\ 1\le r\le 14,\label{congcondh111}\\
&h^{(0)}_{\beta\gamma}(m,n,r),\ 1\le m\le 14,\ 1\le r\le 14.\label{congcondh000}
\end{align}
\end{theorem}

\begin{proof}
Reexamining (\ref{reliancenm5a3b}), (\ref{reliancenm5a2b}), (\ref{reliancenm5b}), we already know that 
\begin{align*}
\pi^{(\alpha)}_{\beta\gamma}(m,r-l) + \phi(l) \ge \pi^{(\alpha)}_{\beta\gamma}(m,r) + 1,
\end{align*} so that in (\ref{reliancenm5a3b}), (\ref{reliancenm5a2b}) we end up with at least 1 additional power of 7.  This already tells us that
\begin{align*}
h^{(\alpha)}_{\beta\gamma}(m,n,r)\equiv \hat{v}(7,0) h^{(\alpha)}_{\beta,\gamma}(m,n-7,r)+ \hat{v}(14,0) h^{(\alpha)}_{\beta,\gamma}(m,n-14,r)\pmod{7}.
\end{align*}  As $\hat{v}(7,0)=25398809\equiv 2\pmod{49}$ and $\hat{v}(14,0)=-1$, if we have 
\begin{align*}
h^{(\alpha)}_{\beta\gamma}(m,n-7,r)\equiv h^{(\alpha)}_{\beta\gamma}(m,n-14,r)\pmod{7},
\end{align*} then the congruence will persist.  We go even further in our Mathematica supplement and show that for the restrictions in (\ref{congcondh100})-(\ref{congcondh000}), we have
\begin{align*}
h^{(\alpha)}_{\beta\gamma}(m,n,r)\equiv h^{(\alpha)}_{\beta\gamma}(m,n+1,r)\pmod{7}.
\end{align*}  That is, for a fixed $m,r$ in our range, $h^{(\alpha)}_{\beta\gamma}$ keeps the same congruence value mod 7 no matter how we vary $n$. Thus, we certainly have the congruence mod 7.

To account for the additional power of 7, we examine our sum over $k$ in (\ref{reliancenm5a3b}).  Notice that we can therefore factor out our power $7^{\pi^{(\alpha)}_{\beta\gamma}(m,r-l)+\phi(l)}$:
\begin{align*}
&\sum_{\substack{0\le\gamma\le 1,\\ 0\le l\le 84,\\ r\ge 1+l-\gamma}}\sum_{\substack{1\le k\le 14,\\ k\neq 7,14}} \hat{v}(k,l)\cdot h^{(\alpha)}_{\beta,\gamma}(m,n-k,r-l)\cdot 7^{\pi^{(\alpha)}_{\beta,\gamma}(m,r-l) + \phi(l)}\\
&=\sum_{\substack{0\le\gamma\le 1,\\ 0\le l\le 84,\\ r\ge 1+l-\gamma}} 7^{\pi^{(\alpha)}_{\beta,\gamma}(m,r-l) + \phi(l)} \sum_{\substack{1\le k\le 14,\\ k\neq 7,14}} \hat{v}(k,l)\cdot h^{(\alpha)}_{\beta,\gamma}(m,n-k,r-l).
\end{align*}  If we take the internal sum over $k$ while holding $r$, $l$ fixed, we have
\begin{align*}
&\sum_{\substack{1\le k\le 14,\\ k\neq 7,14}} \hat{v}(k,l)\cdot h^{(\alpha)}_{\beta,\gamma}(m,n-k,r-l) \equiv h^{(\alpha)}_{\beta,\gamma}(m,n,r-l) \sum_{\substack{1\le k\le 14,\\ k\neq 7,14}} \hat{v}(k,l)\pmod{7},
\end{align*} since $h^{(\alpha)}_{\beta\gamma}$ has the same congruence value mod 7. Finally, we compute in our Mathematica supplement that for any fixed $l$, we have
\begin{align*}
&\sum_{\substack{1\le k\le 14,\\ k\neq 7,14}} \hat{v}(k,l)\equiv 0 \pmod{7}.
\end{align*}  Similarly, we compute that for any fixed $l$,
\begin{align*}
&\hat{v}(7,l)+\hat{v}(14,l)\equiv 0 \pmod{7}.
\end{align*}  So we know that the sums in (\ref{reliancenm5a3b}), (\ref{reliancenm5a2b}) are both divisible by 49 (in addition to the factor $7^{\pi^{(\alpha)}_{\beta\gamma}(m,r)}$), and that therefore
\begin{align*}
h^{(\alpha)}_{\beta\gamma}(m,n,r)\equiv 2 h^{(\alpha)}_{\beta,\gamma}(m,n-7,r) - h^{(\alpha)}_{\beta,\gamma}(m,n-14,r)\pmod{49}.
\end{align*}  We therefore need to directly check 
\begin{align*}
h^{(\alpha)}_{\beta\gamma}(m,n,r)\equiv h^{(\alpha)}_{\beta,\gamma}(m,n+7,r)\pmod{49}.
\end{align*} for $m,r$ in our range, and $1\le n\le 7$.  We check this computationally in the supplement.
\end{proof}

\subsection{Initial Values}\label{sectioninitialetc}

The relations of Theorem \ref{thmrelAA} for $1\le m\le 14$, $1\le n\le 14$ may be directly checked computationally.  Notably, these 196 different relations may be reduced to a smaller number of relations in the following manner:
\begin{align*}
U^{(\alpha)}\left( \frac{y^{\beta}x^{m}}{(1+7x)^{n}} \right) =& \frac{1}{7^m}\cdot U^{(\alpha)}\left( \frac{y^{\beta}(z-1)^{m}}{z^{n}} \right)\\
=& \frac{1}{7^m}\sum_{r=0}^{m}(-1)^{m-r}{{m}\choose{r}}\cdot U^{(\alpha)}\left( z^{r-n} \right)\\
=& \frac{1}{7^m}\sum_{r=0}^{m}(-1)^{m-r}{{m}\choose{r}}\cdot U^{(\alpha)}\left( (1+7x)^{r-n} \right).
\end{align*}  Moreover, by Theorem \ref{modeqnz}, we have
\begin{align*}
U^{(\alpha)}\left( y^{\beta}(1+7x)^n \right) &= -\sum_{k=0}^{13} b_k(z(\tau))\cdot U^{(\alpha)}\left( y^{\beta}(1+7x)^{k+n-14} \right).
\end{align*}  This allows us to ultimately express $U^{(\alpha)}\left( \frac{y^{\beta}x^{m}}{(1+7x)^{n}} \right)$ in terms of $U^{(\alpha)}\left( y^{\beta}(1+7x)^{n} \right)$ for positive $n$.  Of course, for positive $n$ we have
\begin{align*}
U^{(\alpha)}\left( y^{\beta}(1+7x)^n \right) &= \sum_{k=0}^n {{n}\choose{k}}\cdot 7^k\cdot U^{(\alpha)}\left( y^{\beta}x^k \right).
\end{align*}  Finally, using Corollary \ref{cor1},
\begin{align*}
U^{(\alpha)}\left( y^{\beta}x^m \right) &= -\sum_{j=0}^{13} a_j(\tau)\cdot U^{(\alpha)}\left( y^{\beta}x^{m+j-14} \right).
\end{align*}  Therefore, if we know $U^{(\alpha)}\left( y^{\beta}z^k \right)$ for 14 consecutive values of $k$, and of course for $\beta\in\{0,1\}$, then we can utilize Theorems \ref{modeqnx}, \ref{modeqnz} to construct our more general relations.

We have the full construction in our Mathematica supplement, which may be found online at \url{https://www3.risc.jku.at/people/nsmoot/d7congsuppG.nb}.  Each of our fundamental relations will have the form

\begin{align}
U^{(\alpha)}\left( y^{\beta}z^k \right) &= f(\alpha,\beta,k,1/z,y)\text{, for } -6\le k\le -1,\label{initialuaybxketc}\\
U^{(\alpha)}\left( y^{\beta}z^k \right) &= f(\alpha,\beta,k,z,y)\text{, for } 0\le k\le 7,\label{initialuaybxkaetc}\\
\text{for some }f(\alpha,\beta,k,z,y)&\in\mathbb{C}[z]\oplus y\mathbb{C}[z].\label{fdefnetc}
\end{align}  We list these in the opening section of our supplement.

We first consider (\ref{initialuaybxketc}).  Recalling the definition of $U^{(\alpha)}$, we have
\begin{align*}
U_7\left(\mathcal{A}^{1-\alpha} y^{\beta}(\tau)z(\tau)^k \right)=f(\alpha,\beta,k,1/z,y).
\end{align*}  There are various different ways to establish this equality, but we will use a straightforward approach: we will multiply both sides by an eta quotient which will induce a pole at $[\infty]$ and annihilate all other poles.  We then have only to compare the principal parts and constants of either side.  If they match, then we must have full equality.  In our supplement we will justify many of our formulae using results from \cite{Radu}.

We have $\mathcal{A}^{1-\alpha} y^{\beta}(\tau)z(\tau)^k\in\mathcal{M}\left(  \Gamma_0(98)\right)$.  Notice that we work with $z$ rather than $x$, because we have theorems which determine the poles and zeros of eta quotients, e.g., \cite[Theorem 23]{Radu}.

In our Mathematica supplement, we show that for $f$ defined as in (\ref{fdefnetc}), if
\begin{align*}
\mathfrak{m}(\tau) := \frac{1}{q^{4}}\frac{(q^2;q^2)_{\infty}^{5}(q^7;q^7)_{\infty}^{7}}{(q;q)_{\infty}(q^{14};q^{14})_{\infty}^{11}}\text{, then}
\end{align*}
\begin{align*}
\mathfrak{m}(\tau)^{45}&\cdot f(\alpha,\beta,k,1/z,y)\in\mathcal{M}^{\infty}\left( \Gamma_0(14) \right),\\
\mathfrak{m}(7\tau)^{45}&\cdot \mathcal{A}^{1-\alpha} y^{\beta}(\tau)z(\tau)^k\in\mathcal{M}^{\infty}\left( \Gamma_0(98) \right).
\end{align*}  We thus want to prove the following:
\begin{align}
U_7\left(\mathfrak{m}(7\tau)^{45}\mathcal{A}^{1-\alpha} y^{\beta}(\tau)z(\tau)^k\right)=\mathfrak{m}(\tau)^{45}f(\alpha,\beta,k,1/z,y),\text{ }-6\le k\le -1.\label{fundrelmodproveA}
\end{align}  We can directly compute and compare the principal parts and constants of either side to show that they match, thus confirming the cases.  From these relations, we systematically construct and verify the 196 initial relations we need for Theorem \ref{thmrelAA}.

For (\ref{initialuaybxkaetc}), it is easier to work with $1/z$ than with $z$.  The fundamental relations have a maximum positive $z$-power of 71.  We therefore divide both sides of the relations by $z^{71}$, and we have
\begin{align*}
U_7\left(\mathcal{A}^{1-\alpha} y^{\beta}(\tau)z(\tau)^k z(7\tau)^{-71} \right)=z^{-71}f(\alpha,\beta,k,z,y)\in\mathbb{C}[1/z]\oplus y\mathbb{C}[1/z].
\end{align*}  Taking $\mathfrak{m}$ again, we show in our supplement that
\begin{align*}
\mathfrak{m}(\tau)^{74}&\cdot z^{-71}f(\alpha,\beta,k,z,y)\in\mathcal{M}^{\infty}\left( \Gamma_0(14) \right),\\
\mathfrak{m}(7\tau)^{74}&\cdot \mathcal{A}^{1-\alpha} y^{\beta}(\tau)z(\tau)^k z(7\tau)^{-71}\in\mathcal{M}^{\infty}\left( \Gamma_0(98) \right).
\end{align*}  We thus want to prove the following:
\begin{align}
U_7\left(\mathfrak{m}(7\tau)^{74}\cdot \mathcal{A}^{1-\alpha} y^{\beta}(\tau)z(\tau)^k z(7\tau)^{-71}\right) =\mathfrak{m}(\tau)^{74}\cdot z^{-71}f(\alpha,\beta,k,z,y),\text{ }0\le k\le 7.\label{fundrelmodproveB}
\end{align}  Proving (\ref{fundrelmodproveA}), (\ref{fundrelmodproveB}) therefore confirms our fundamental relations.

We use a similar approach to prove Theorem \ref{modeqnz}, also in the supplement.  Because of the size of each relation, we do not include them here.  Every relation, and every step in our construction, is included in our supplement.

\section{Partial Stability: $\mathcal{V}^{(0)}_n\rightarrow\mathcal{V}^{(1)}_{7n+3}$}\label{section0to1full}

We can now begin to prove Theorem \ref{Main}.  We have the following important result:

\begin{theorem}\label{goingv0tov1}
Let $f\in\mathcal{V}^{(0)}_n$.  Then $U^{(0)}(f)\in\mathcal{V}^{(1)}_{7n+3}.$
\end{theorem}

Thus, applying $U^{(0)}$ to a function like $L_{2\alpha}$ will result in a function which, at least superficially, resembles the predicted form of $L_{2\alpha+1}$.

\begin{proof}

We let $f\in\mathcal{V}^{(0)}_n$ by hypothesis.  Let us denote
\begin{align*}
f=\frac{1}{(1+7x)^n}\sum_{\substack{0\le\beta\le 1\\ m\ge 1-\beta}}s_{\beta}(m)7^{\theta^{(0)}_{\beta}(m)}y^{\beta}x^m.
\end{align*}  
Applying $U^{(0)}$ and remembering that the operator is linear, using Theorem \ref{thmrelAA}, we have
\begin{align}
U^{(0)}\left( f \right) =& \frac{1}{(1+7x)^{7n+3}}\sum_{\substack{0\le\beta\le 1\\ m\ge 1-\beta\\ 0\le\gamma\le 1\\ r\ge 1-\gamma}} s_{\beta}(m)h^{(0)}_{\beta\gamma}(m,n,r) 7^{\pi^{(0)}_{\beta\gamma}(m,r)+\theta^{(0)}_{\beta}(m)}y^{\gamma}x^r\label{u23fsad}\\
=& \frac{1}{(1+7x)^{7n+3}}\sum_{\substack{0\le\gamma\le 1\\ r\ge 1-\gamma}} t_{\gamma}(r) 7^{\theta^{(1)}_{\gamma}(r)}y^{\gamma}x^r,
\end{align} for
\begin{align*}
t_{\gamma}(r) =& \sum_{\substack{0\le\beta\le 1\\ m\ge 1-\beta}} s_{\beta}(m)h^{(0)}_{\beta\gamma}(m,n,r) 7^{\pi^{(0)}_{\beta\gamma}(m,r)+\theta^{(0)}_{\beta}(m)-\theta^{(1)}_{\gamma}(r)}.
\end{align*}  It remains for us to check that
\begin{align*}
\pi^{(0)}_{\beta\gamma}(m,r)+\theta^{(0)}_{\beta}(m)\ge \theta^{(1)}_{\gamma}(r),
\end{align*} which confirms membership in $\mathcal{V}^{(1)}_{7n+3}$.  This is verified in the subsections below.

\end{proof}

\subsection{$\displaystyle{\pi^{(0)}_{00}(m,r)+\theta^{(0)}_0(m)\ge \theta^{(1)}_0(r)}$}

Here our bounds are $m\ge 1$, $r\ge 1$.\\

\begin{enumerate}
\item $r\ge 4$, $m\ge 4$:
\begin{align*}
&\left\lfloor \frac{7r-m-23}{9} \right\rfloor + \left\lfloor \frac{7m-16}{9} \right\rfloor\ge \left\lfloor \frac{7r+6m-47}{9} \right\rfloor\\
&\ge \left\lfloor \frac{7r-28}{9} \right\rfloor + \left\lfloor \frac{6m-19}{9} \right\rfloor\\
&\ge \theta^{(1)}_0(r)
\end{align*}
\item $r\ge 4$, $1\le m\le 3$:
\begin{align*}
&\left\lfloor \frac{7r-m-23}{9} \right\rfloor + 0\ge\left\lfloor \frac{7r-26}{9} \right\rfloor\\
&\ge \theta^{(1)}_0(r)
\end{align*}
\item $r=3$, $m\ge 3$:
\begin{align*}
&-1 + \left\lfloor \frac{7m-16}{9} \right\rfloor\ge -1\\
&=\theta^{(1)}_0(r)
\end{align*}
\item $1\le r\le 3$, $1\le m\le 2$:
\begin{align*}
&-1 + \theta^{(0)}_0(m)\ge -1\\
&= \theta^{(1)}_0(r)
\end{align*}
\end{enumerate}

\subsection{$\displaystyle{\pi^{(0)}_{01}(m,r)+\theta^{(0)}_0(m)\ge\theta^{(1)}_1(r)}$}

Here our bounds are $m\ge 1$, $r\ge 0$.\\

\begin{enumerate}
\item $r\ge 4$, $m\ge 4$:
\begin{align*}
&\left\lfloor \frac{7r-m-9}{9} \right\rfloor + \left\lfloor \frac{7m-16}{9} \right\rfloor\ge\left\lfloor \frac{7r+6m-33}{9} \right\rfloor\\
&\ge \left\lfloor \frac{7r-14}{9} \right\rfloor + \left\lfloor \frac{6m-19}{9} \right\rfloor\\
&\ge \theta^{(1)}_1(r)
\end{align*}
\item $r\ge 4$, $1\le m\le 3$:
\begin{align*}
&\left\lfloor \frac{7r-m-9}{9} \right\rfloor + 0\ge \left\lfloor \frac{7r-12}{9} \right\rfloor\\
&\ge \theta^{(1)}_1(r)
\end{align*}
\item $2\le r\le 3$, $m\ge 4$:
\begin{align*}
&\left\lfloor \frac{7r-m-9}{9} \right\rfloor + \left\lfloor \frac{7m-16}{9} \right\rfloor\ge\left\lfloor \frac{7r+6m-33}{9} \right\rfloor\\
&\ge \left\lfloor \frac{7r-14}{9} \right\rfloor + \left\lfloor \frac{6m-19}{9} \right\rfloor\\
&\ge \theta^{(1)}_1(r)
\end{align*}
\item $2\le r\le 3$, $1\le m\le 3$:
\begin{align*}
&\left\lfloor \frac{7r-m-9}{9} \right\rfloor + 0\ge \left\lfloor \frac{7r-12}{9} \right\rfloor\\
&\ge\theta^{(1)}_1(r)
\end{align*}
\item $0\le r\le 1$, $m\ge 1$:
\begin{align*}
-1 + \theta^{(0)}_0(m)\ge -1 = \theta^{(1)}_1(r)
\end{align*}
\end{enumerate}

\subsection{$\displaystyle{\pi^{(0)}_{10}(m,r)+\theta^{(0)}_1(m)\ge\theta^{(1)}_0(r)}$}

Here our bounds are $m\ge 0$, $r\ge 1$.\\

\begin{enumerate}
\item $r\ge 4$, $m\ge 2$:
\begin{align*}
&\left\lfloor \frac{7r-m-27}{9} \right\rfloor + \left\lfloor \frac{7m-5}{9} \right\rfloor\ge \left\lfloor \frac{7r+6m-40}{9} \right\rfloor\\
&\ge\left\lfloor \frac{7r-28}{9} \right\rfloor + \left\lfloor \frac{6m-12}{9} \right\rfloor\\
&\ge \theta^{(1)}_0(r)
\end{align*}
\item $r\ge 4$, $m=1$:
\begin{align*}
\left\lfloor \frac{7r-28}{9} \right\rfloor + 0\ge \theta^{(1)}_0(r)
\end{align*}
\item $r\ge 4$, $m=0$:
\begin{align*}
\left\lfloor \frac{7r-28}{9} \right\rfloor + 0 \ge \theta^{(1)}_0(r)
\end{align*}
\item $1\le r\le 3$, $m\ge 0$:
\begin{align*}
-1 + \theta^{(0)}_1(m)\ge -1 = \theta^{(1)}_0(r).
\end{align*}
\end{enumerate}

\subsection{$\displaystyle{\pi^{(0)}_{11}(m,r)+\theta^{(0)}_1(m)\ge\theta^{(1)}_1(r)}$}

Here our bounds are $m\ge 0$, $r\ge 0$.\\

\begin{enumerate}
\item $r\ge 2$, $m\ge 3$:
\begin{align*}
&\left\lfloor \frac{7r-m-13}{9} \right\rfloor + \left\lfloor \frac{7m-5}{9} \right\rfloor\ge \left\lfloor \frac{7r-6m-26}{9} \right\rfloor\\
&\ge\left\lfloor \frac{7r-14}{9} \right\rfloor + \left\lfloor \frac{6m-12}{9} \right\rfloor\\
&\ge\theta^{(1)}_1(r)
\end{align*}
\item $r\ge 2$, $m=2$:
\begin{align*}
&\left\lfloor \frac{7r-15}{9} \right\rfloor + 1\ge \left\lfloor \frac{7r-6}{9} \right\rfloor\\
&\ge\theta^{(1)}_1(r)
\end{align*}
\item $r\ge 2$, $0\le m\le 1$:
\begin{align*}
&\left\lfloor \frac{7r-m+5}{9} \right\rfloor + 0\ge \left\lfloor \frac{7r+4}{9} \right\rfloor\\
&\ge \theta^{(1)}_1(r)
\end{align*}
\item $0\le r\le 1$, $m\ge 0$:
\begin{align*}
-1 + \theta^{(0)}_1(m)\ge -1 = \theta^{(1)}_1(r)
\end{align*}
\end{enumerate}

\section{Partial Stability: $\mathcal{V}^{(1)}_n\rightarrow\mathcal{V}^{(0)}_{7n}$}\label{section1to0part}

We now come to the more difficult step.  We would like to prove a result analogous to that of Theorem \ref{goingv0tov1}.  Suppose we have some

\begin{align*}
f\in\mathcal{V}^{(1)}_n.
\end{align*}  We want to prove that $U^{(1)}(f)\in\mathcal{V}^{(1)}_{7n}$.  Again, we can denote
\begin{align*}
f=\frac{1}{(1+7x)^n}\sum_{\substack{0\le\beta\le 1\\ m\ge 1-\beta}}s_{\beta}(m)7^{\theta^{(1)}_{\beta}(m)}y^{\beta}x^m,
\end{align*} and upon applying $U^{(1)}$,
\begin{align*}
U^{(1)}\left( f \right) =& \frac{1}{(1+7x)^{7n}}\sum_{\substack{0\le\beta\le 1\\ m\ge 1-\beta\\ 0\le\gamma\le 1\\ r\ge 1-\gamma}} s_{\beta}(m)h^{(1)}_{\beta\gamma}(m,n,r) 7^{\pi^{(1)}_{\beta\gamma}(m,r)+\theta^{(1)}_{\beta}(m)}y^{\gamma}x^r\\
=& \frac{1}{(1+7x)^{7n}}\sum_{\substack{0\le\gamma\le 1\\ r\ge 1-\gamma}} t_{\gamma}(r) 7^{\theta^{(0)}_{\gamma}(r)}y^{\gamma}x^r,
\end{align*} for
\begin{align}
t_{\gamma}(r) =& \sum_{\substack{0\le\beta\le 1\\ m\ge 1-\beta}} s_{\beta}(m)h^{(1)}_{\beta\gamma}(m,n,r) 7^{\pi^{(1)}_{\beta\gamma}(m,r)+\theta^{(1)}_{\beta}(m)-\theta^{(0)}_{\gamma}(r)}.\label{tgammadefn}
\end{align}  Ideally, we would want to confirm that
\begin{align*}
\pi^{(1)}_{\beta\gamma}(m,r)+\theta^{(1)}_{\beta}(m)\ge \theta^{(0)}_{\gamma}(r) + 1.
\end{align*}  This will \textit{not} be true for all cases.  However, we can bound the counterexamples to a finite set of $(m,r)$.

\subsection{$\displaystyle{\pi^{(1)}_{00}(m,r)+\theta^{(1)}_0(m)\ge \theta^{(0)}_0(r)+1}$}\label{100bounds}

Here our bounds are $m\ge 1$, $r\ge 1$.\\

\begin{enumerate}
\item $r\ge 3$, $m\ge 5$:
\begin{align*}
&\left\lfloor \frac{7r-m}{9} \right\rfloor + \left\lfloor \frac{7m-28}{9} \right\rfloor\ge \left\lfloor \frac{7r-6m-36}{9} \right\rfloor\\
&\ge\left\lfloor \frac{7r-16}{9} \right\rfloor + \left\lfloor \frac{6m-20}{9} \right\rfloor\\
&\ge\theta^{(0)}_0(r)+1
\end{align*}
\item $r\ge 3$, $m=4$:
\begin{align*}
&\left\lfloor \frac{7r-4}{9} \right\rfloor + 0\ge \left\lfloor \frac{7r-13}{9} \right\rfloor + 1\\
&\ge\theta^{(0)}_0(r)+1
\end{align*}
\item $r\ge 9$, $1\le m\le 3$:
\begin{align*}
&\left\lfloor \frac{7r+2}{9} \right\rfloor -1 \ge \left\lfloor \frac{7r-7}{9} \right\rfloor\\
&\ge\theta^{(0)}_0(r)+1.
\end{align*}
\item $r=2$, $m\ge 6$:
\begin{align*}
&\left\lfloor \frac{14-m}{9} \right\rfloor + \left\lfloor \frac{7m-28}{9} \right\rfloor\ge\left\lfloor \frac{6m-22}{9} \right\rfloor \\
&\ge\theta^{(0)}_0(2)+1.
\end{align*}
\item $r=2$, $4\le m\le 5$:
\begin{align*}
&\left\lfloor \frac{14-m}{9} \right\rfloor + \left\lfloor \frac{7m-28}{9} \right\rfloor \ge 0 + 1 \\
&\ge\theta^{(0)}_0(2)+1.
\end{align*}
\item $r=1$, $m\ge 6$:
\begin{align*}
&\left\lfloor \frac{7-m}{9} \right\rfloor + \left\lfloor \frac{7m-28}{9} \right\rfloor \ge 0 + 1 \\
&\ge\theta^{(0)}_0(1)+1.
\end{align*}
\end{enumerate}  Our deviant cases are:

\begin{itemize}
\item $2\le r\le 8$, $1\le m\le 3$,
\item $r=1$, $1\le m\le 5$.
\end{itemize}

\subsection{$\displaystyle{\pi^{(1)}_{01}(m,r)+\theta^{(1)}_0(m)\ge \theta^{(0)}_1(r)+1}$}

Here our bounds are $m\ge 1$, $r\ge 0$.\\

\begin{enumerate}
\item $r\ge 2$, $m\ge 4$:
\begin{align*}
&\left\lfloor \frac{7r-m+16}{9} \right\rfloor + \left\lfloor \frac{7m-28}{9} \right\rfloor\ge\left\lfloor \frac{7r+6m-20}{9} \right\rfloor\\
&\ge\left\lfloor \frac{7r-5}{9} \right\rfloor + \left\lfloor \frac{6m-15}{9} \right\rfloor\\
&\ge\theta^{(0)}_1(r)+1
\end{align*}
\item $r\ge 2$, $1\le m\le 3$:
\begin{align*}
&\left\lfloor \frac{7r-m+16}{9} \right\rfloor -1\\
&\ge\left\lfloor \frac{7r-3+7}{9} \right\rfloor=\left\lfloor \frac{7r+4}{9} \right\rfloor\\
&\ge\left\lfloor \frac{7r-5}{9} \right\rfloor + 1\\
&\ge\theta^{(0)}_1(r)+1
\end{align*}
\item $r=1$, $m\ge 4$:
\begin{align*}
&\left\lfloor \frac{23-m}{9} \right\rfloor + \left\lfloor \frac{7m-28}{9} \right\rfloor\ge 0+\left\lfloor \frac{6m-13}{9} \right\rfloor\\
&\ge\theta^{(0)}_1(1)+1
\end{align*}
\item $r=1$, $1\le m\le 3$:
\begin{align*}
&\left\lfloor \frac{23-m}{9} \right\rfloor -1 \ge \left\lfloor \frac{20}{9} \right\rfloor - 1\\
&= 2-1\\
&\ge\theta^{(0)}_1(1)+1
\end{align*}
\item $r=0$, $m\ge 5$:
\begin{align*}
&\left\lfloor \frac{16-m}{9} \right\rfloor + \left\lfloor \frac{7m-28}{9} \right\rfloor\ge 0+\left\lfloor \frac{6m-20}{9} \right\rfloor\\
&\ge\theta^{(0)}_1(0)+1
\end{align*}
\item $r=0$, $m=4$:
\begin{align*}
&\left\lfloor \frac{12}{9} \right\rfloor + 0 = 1\\
&=\theta^{(0)}_1(0)+1
\end{align*}
\end{enumerate}

The deviant cases are

\begin{itemize}
\item $r=0$, $1\le m\le 3$.
\end{itemize}

\subsection{$\displaystyle{\pi^{(1)}_{10}(m,r)+\theta^{(1)}_1(m)\ge \theta^{(0)}_0(r)+1}$}\label{110bounds}

Here our bounds are $m\ge 0$, $r\ge 1$.\\

\begin{enumerate}
\item $r\ge 3$, $m\ge 3$:
\begin{align*}
&\left\lfloor \frac{7r-m-2}{9} \right\rfloor + \left\lfloor \frac{7m-14}{9} \right\rfloor\ge \left\lfloor \frac{7r+6m-24}{9} \right\rfloor \\
&\left\lfloor \frac{7r-16}{9} \right\rfloor + \left\lfloor \frac{6m-8}{9} \right\rfloor\\
&\ge\theta^{(0)}_0(r)+1
\end{align*}
\item $r\ge 3$, $m=2$:
\begin{align*}
&\left\lfloor \frac{7r-4}{9} \right\rfloor + 0\ge \left\lfloor \frac{7r-7}{9} \right\rfloor\\
&=\theta^{(0)}_0(r)+1
\end{align*}
\item $r\ge 11$, $0\le m\le 1$:
\begin{align*}
&\left\lfloor \frac{7r+2}{9} \right\rfloor - 1 = \left\lfloor \frac{7r-7}{9} \right\rfloor\\
&=\theta^{(0)}_0(r)+1
\end{align*}
\item $r=2$, $m\ge 4$:
\begin{align*}
&\left\lfloor \frac{12-m}{9} \right\rfloor + \left\lfloor \frac{7m-14}{9} \right\rfloor\ge \left\lfloor \frac{6m-10}{9} \right\rfloor \\
&\ge 1\\
&=\theta^{(0)}_0(2)+1
\end{align*}
\item $r=2$, $m=2$:
\begin{align*}
&\left\lfloor \frac{10}{9} \right\rfloor + 0= 1 \\
&=\theta^{(0)}_0(2)+1
\end{align*}
\item $r=1$, $m\ge 4$:
\begin{align*}
&\left\lfloor \frac{5-m}{9} \right\rfloor + \left\lfloor \frac{7m-13}{9} \right\rfloor= 0+ \left\lfloor \frac{7m-13}{9} \right\rfloor \\
&\ge 1\\
&=\theta^{(0)}_0(1)+1.
\end{align*}
\end{enumerate}

The deviant cases are
\begin{itemize}
\item $r=1$, $0\le m\le 3$,
\item $r=2$, $m=3$,
\item $2\le r\le 10$, $0\le m\le 1$.
\end{itemize}

\subsection{$\displaystyle{\pi^{(1)}_{11}(m,r)+\theta^{(1)}_1(m)\ge \theta^{(0)}_1(r)+1}$}

Here our bounds are $m\ge 0$, $r\ge 0$.\\

\begin{enumerate}
\item $r\ge 2$, $m\ge 3$:
\begin{align*}
&\left\lfloor \frac{7r-m+3}{9} \right\rfloor + 1 + \left\lfloor \frac{7m-14}{9} \right\rfloor\ge\left\lfloor \frac{7r+6m-10}{9} \right\rfloor\\
&\ge\left\lfloor \frac{7r-5}{9} \right\rfloor + \left\lfloor \frac{6m-5}{9} \right\rfloor\\
&\ge\theta^{(0)}_1(r)+1
\end{align*}
\item $r\ge 2$, $m=2$:
\begin{align*}
&\left\lfloor \frac{7r+1}{9} \right\rfloor + 1 + 0 = \left\lfloor \frac{7r+10}{9} \right\rfloor\\
&\ge\theta^{(0)}_1(r)+1
\end{align*}
\item $r\ge 7$, $0\le m\le 1$:
\begin{align*}
&\left\lfloor \frac{7r+4}{9} \right\rfloor + 1 - 1\\
&= \left\lfloor \frac{7r+4}{9} \right\rfloor\\
&= \theta^{(0)}_1(r)+1
\end{align*}
\item $r=1$, $m\ge 2$:
\begin{align*}
&\left\lfloor \frac{7r-m+3}{9} \right\rfloor + 1 +\left\lfloor \frac{7m-14}{9} \right\rfloor\ge 1\\
&= \theta^{(0)}_1(1)+1
\end{align*}
\end{enumerate}

The deviant cases are
\begin{itemize}
\item $1\le r\le 6$, $0\le m\le 1$.
\end{itemize}

\section{Congruence Ideal Stability}\label{sectionstable} 

Let us review our current situation.  We have a sequence $\left( L_{\alpha} \right)_{\alpha\ge 1}$ of modular functions, together with an alternating sequence of operators $U^{(1)}$, $U^{(0)}$ which take each $L_{\alpha}$ to $L_{\alpha+1}$.  The odd-indexed elements are apparently members of a set $\mathcal{V}^{(1)}$, and the even-indexed elements members of $\mathcal{V}^{(0)}$.  Generally, an element of $\mathcal{V}^{(\alpha)}$ should have the form
\begin{align*}
f =\frac{1}{(1+7x)^n}\sum_{\substack{0\le\beta\le 1\\ m\ge 1-\beta}}s_{\beta}(m)7^{\theta^{(\alpha)}_{\beta}(m)}y^{\beta}x^m.
\end{align*}  We understand that applying $U^{(0)}$ to an element of $\mathcal{V}^{(0)}$ (e.g., $L_{2\alpha}$) produces an element of $\mathcal{V}^{(1)}$.  Moreover, we want the application of $U^{(1)}$ to an element of $\mathcal{V}^{(1)}$ (e.g., $L_{2\alpha-1}$) to produce an element of $\mathcal{V}^{(0)}$ \textit{with an additional factor of} 7.  However, we know that this is not generally true: $U^{(1)}(f)$ will produce a rational polynomial in which most monomials will gain the necessary power of 7, but not all.  We are thus forced to conclude that these deviant monomials must somehow cancel out.  We must therefore examine the various exceptional cases and find a way to keep track of the complex behavior between the coefficients $s_{\beta}(m)$ which ensures this cancelation.

To illustrate how this can be done, we will recall our notation.  We let
\begin{align*}
U^{(1)}\left( f \right) =& \frac{1}{(1+7x)^{7n}}\sum_{\substack{0\le\beta\le 1\\ m\ge 1-\beta\\ 0\le\gamma\le 1\\ r\ge 1-\gamma}} s_{\beta}(m)h^{(1)}_{\beta\gamma}(m,n,r) 7^{\pi^{(1)}_{\beta\gamma}(m,r)+\theta^{(1)}_{\beta}(m)}y^{\gamma}x^r\\
=& \frac{1}{(1+7x)^{7n}}\sum_{\substack{0\le\gamma\le 1\\ r\ge 1-\gamma}} t_{\gamma}(r) 7^{\theta^{(0)}_{\gamma}(r)}y^{\gamma}x^r,
\end{align*} and
\begin{align*}
t_{\gamma}(r) =& \sum_{\substack{0\le\beta\le 1\\ m\ge 1-\beta}} s_{\beta}(m)h^{(1)}_{\beta\gamma}(m,n,r) 7^{\pi^{(1)}_{\beta\gamma}(m,r)+\theta^{(1)}_{\beta}(m)-\theta^{(0)}_{\gamma}(r)}.
\end{align*}  We first consider $t_0(1)$.  By Sections \ref{100bounds} and \ref{110bounds}, we have the following expression:
\begin{align*}
t_0(1) =& \frac{h^{(1)}_{00}(1,n,1)}{7} s_0(1) + \frac{h^{(1)}_{00}(2,n,1)}{7} s_0(2) + \frac{h^{(1)}_{00}(3,n,1)}{7} s_0(3) + h^{(1)}_{00}(4,n,1) s_0(4) + h^{(1)}_{00}(5,n,1) s_0(5)\\ & + \frac{h^{(1)}_{10}(0,n,1)}{7} s_1(0) + \frac{h^{(1)}_{10}(1,n,1)}{7} s_1(1) + h^{(1)}_{10}(2,n,1) s_1(2) +h^{(1)}_{10}(3,n,1) s_1(3) + 7M,
\end{align*} for some $M\in\mathbb{Z}$.  Not only do we need $t_0(1)$ to be an integer, but we need it to be a multiple of 7.  As such, we multiply through by 7 to kill the denominators, and we have the following integer relation:
\begin{align*}
&h^{(1)}_{00}(1,n,1) s_0(1) + h^{(1)}_{00}(2,n,1) s_0(2) + h^{(1)}_{00}(3,n,1) s_0(3) + 7 h^{(1)}_{00}(4,n,1) s_0(4) + 7 h^{(1)}_{00}(5,n,1) s_0(5)\\ & + h^{(1)}_{10}(0,n,1) s_1(0) + h^{(1)}_{10}(1,n,1) s_1(1) + 7 h^{(1)}_{10}(2,n,1) s_1(2) +7 h^{(1)}_{10}(3,n,1) s_1(3) \equiv 0\pmod{49}.
\end{align*}  Recall that $n\equiv 3\pmod{7}$, and all of the coefficients are subject to Theorem \ref{hcongred37}.  We can therefore restrict ourselves to working with $h^{(1)}_{\beta\gamma}(m,3,1)$, which we can readily compute for the finite number of necessary values of $m$:
\begin{align}
43 s_0(1) + 48 s_0(2) + 15 s_0(3) + 42 s_0(4) + 7 s_0(5) + 8 s_1(0) + 27 s_1(1) + 7 s_1(2) + 42 s_1(3)\equiv 0\pmod{49}.\label{idealel1}
\end{align}  If this relation is satisfied, then $t_0(1)$ will be an integer divisible by 7, as we need.

Examining (\ref{L1expressxy}) by taking $f=L_1$, we find that
\begin{align*}
s_0(1) &= 320013737\equiv 29\pmod{49},\\
s_0(2) &= 29164229489\equiv 45\pmod{49},\\
s_0(3) &= 1226655768017\equiv 1\pmod{49},\\
s_0(4) &= 4505536916704\equiv 47\pmod{49},\\
s_0(5) &= 79044206825472\equiv 42\pmod{49},\\
s_1(0) &= -320013688\equiv 20\pmod{49},\\
s_1(1) &= -28844055074\equiv 40\pmod{49},\\
s_1(2) &= -171156188528\equiv 32\pmod{49},\\
s_1(3) &= -4337927987008\equiv 0\pmod{49}.
\end{align*}  With these substitutions, we find that
\begin{align}
43 (29) + 48 (45) + 15 (1) + 42 (47) + 7 (42) + 8 (20) + 27 (40) + 7 (32) + 42 (0) = 7154= 146(49).\label{checkingL1forfirst}
\end{align}  This of course makes sense, because $L_2$ is indeed divisible by 7, and is a member of $\mathcal{V}^{(0)}_{21}$.

We can similarly examine $t_{\gamma}(r)$ for the remaining deviant cases.  As it happens, this is the only relation that we will need to take modulo 49.  The remaining necessary relations may be taken modulo powers of 7, as we note in Tables \ref{tablew00}, \ref{tablew11}.

\begin{table}[hbt!]
\begin{center}
\begin{tabular}{l|r}
 $r$      & $t_0(r)\bmod 7$\\
\hline\\
 $2$         & $h^{(1)}_{00}(1,n,2)s_0(1)+h^{(1)}_{00}(2,n,2)s_0(2)+h^{(1)}_{00}(3,n,2)s_0(3)+h^{(1)}_{10}(0,n,2)s_1(0)+h^{(1)}_{10}(1,n,2)s_1(1)$\\
 $3$         & $0$\\
 $4$         & $h^{(1)}_{10}(0,n,4)s_1(0)+h^{(1)}_{10}(1,n,4)s_1(1)$\\
 $5$         & $h^{(1)}_{00}(3,n,5)s_0(3)+h^{(1)}_{10}(0,n,5)s_1(0)$\\
 $6$         & $0$\\
 $7$         & $0$\\
 $8$         & $h^{(1)}_{10}(0,n,8)s_1(0)+h^{(1)}_{10}(1,n,8)s_1(1)$\\
 $9$         & $h^{(1)}_{10}(0,n,9)s_1(0)+h^{(1)}_{10}(1,n,8)s_1(1)$\\
 $10$       & $h^{(1)}_{10}(0,n,10)s_1(0)+h^{(1)}_{10}(1,n,10)s_1(1)$
 \end{tabular}
\vspace{0.2 cm}
\caption{Nonvanishing terms of $t_0(r)\bmod 7$}\label{tablew00}
\end{center}
\end{table}

\begin{table}[hbt!]
\begin{center}
\begin{tabular}{l|r}
 $r$      & $t_1(r)\bmod 7$\\
\hline\\
 $0$         & $h^{(1)}_{01}(3,n,0)s_0(3)+h^{(1)}_{11}(0,n,0)s_1(0)+h^{(1)}_{11}(1,n,0)s_1(1)$\\
 $1$         & $0$\\
 $2$         & $h^{(1)}_{11}(0,n,2)s_1(0)+h^{(1)}_{11}(1,n,2)s_1(1)$\\
 $3$         & $0$\\
 $4$         & $0$\\
 $5$         & $0$\\
 $6$         & $h^{(1)}_{11}(0,n,6)s_1(0)+h^{(1)}_{11}(1,n,6)s_1(1)$
 \end{tabular}
\vspace{0.2 cm}
\caption{Nonvanishing terms of $t_1(r)\bmod 7$}\label{tablew11}
\end{center}
\end{table}  These coefficients are again subject to Theorem \ref{hcongred37}.  We can simplify to the relations in Tables \ref{tablew0}, \ref{tablew1}.

\begin{table}[hbt!]
\begin{center}
\begin{tabular}{l|r}
 $r$      & $t_0(r)\bmod 7$\\
\hline\\
 $2$         & $3s_0(1)+2s_0(2)+2s_0(3)+3s_1(0)+4s_1(1)$\\
 $3$         & $0$\\
 $4$         & $4s_1(0)+5s_1(1)$\\
 $5$         & $5s_0(3)+5s_1(0)$\\
 $6$         & $0$\\
 $7$         & $0$\\
 $8$         & $3s_1(0)+2s_1(1)$\\
 $9$         & $3s_1(0)+2s_1(1)$\\
 $10$       & $s_1(0)+3s_1(1)$
 \end{tabular}
\vspace{0.2 cm}
\caption{Value of $t_0(r)\bmod 7$}\label{tablew0}
\end{center}
\end{table}

\begin{table}[hbt!]
\begin{center}
\begin{tabular}{l|r}
 $r$      & $t_1(r)\bmod 7$\\
\hline\\
 $0$         & $s_0(3)+5s_1(0)+5s_1(1)$\\
 $1$         & $0$\\
 $2$         & $6s_1(0)+4s_1(1)$\\
 $3$         & $0$\\
 $4$         & $0$\\
 $5$         & $0$\\
 $6$         & $s_1(0)+3s_1(1)$
 \end{tabular}
\vspace{0.2 cm}
\caption{Value of $t_1(r)\bmod 7$}\label{tablew1}
\end{center}
\end{table}

We can quickly verify by computer (or even by hand, were we so inclined) that the coefficients $s_{\beta}(m)$ of $L_1$ do indeed satisfy these relations.  For want of space, we verify this in our Mathematica supplement.

Notice, however, that we are not done.  It is not sufficient to verify that $L_1$ satisfies the relations that we have just given.  After all, there is no way of knowing whether $L_3$, $L_5$, or $L_{44847}$ are also so composed such that applying $U^{(1)}$ will cause all deviant monomials to cancel out.  To determine this, we construct what we call the \textit{congruence ideal sequence} associated with the congruence family.  From there, we determine \textit{ideal stability}.

Let us consider a function
\begin{align}
f_{\alpha}=\frac{1}{(1+7x)^n}\sum_{\substack{0\le\beta\le 1\\ m\ge 1-\beta}}s_{\alpha,\beta}(m)7^{\theta^{(1)}_{\beta}(m)}y^{\beta}x^m,\label{falpha}
\end{align} in which $\beta\in\{0,1\}$ as always, but in which $\alpha$ can be any positive integer index.  For our purposes, $f_{\alpha}$ will be associated with $L_{2\alpha-1}/7^{\alpha -1}$.  

It will be useful for us to define the associated vector
\begin{align}
\mathbf{s}_{\alpha} &:= \left( s_{\alpha,0}(1), s_{\alpha,0}(2), s_{\alpha,0}(3), s_{\alpha,0}(4), s_{\alpha,0}(5), s_{\alpha,1}(0), s_{\alpha,1}(1), s_{\alpha,1}(2), s_{\alpha,1}(3) \right).\label{salphavector}
\end{align}  Now we wish to encapsulate the relations in (\ref{idealel1}) and Tables \ref{tablew0} and \ref{tablew1}.  Notice that the relations in Tables \ref{tablew0}, \ref{tablew1} are all mod 7 relations, i.e., we can express them as polynomials in a finite field.  As such, we can define the associated generated set of functions which summarizes our relations in Tables \ref{tablew0} and \ref{tablew1}.  This gives us the following:
\begin{align*}
4X_2 + 6X_3 + 6X_6 + 6X_7,\\
4X_3 + 6X_6 + 6X_7,\\
4X_6 + 5X_7,
\end{align*} in which we use the indeterminate vector $\mathbf{X} = \left( X_1, X_2, X_3, X_4, X_5, X_6, X_7, X_8, X_9 \right)$ in place of $\mathbf{s}_{\alpha}$.  We cannot do as much with (\ref{idealel1}), since it is a mod $7^2$ relation.  Instead, we simply multiply our function elements through by 7 and work over a coefficient ring $\mathbb{Z}/49\mathbb{Z}$.  We then have the ideal
\begin{align*}
I(\mathbf{X}) := &\big( p_1(\mathbf{X}), p_2(\mathbf{X}), p_3(\mathbf{X}), p_4(\mathbf{X}) \big) \le\left(\mathbb{Z}/49\mathbb{Z}\right)[\mathbf{X}],
\end{align*}  in which the polynomials $p_k(\mathbf{X})$ are defined:
\begin{align*}
p_1(\mathbf{X}) &= 43X_1 + 48X_2 + 15X_3 + 42X_4 + 7X_5 + 8X_6 + 27X_7 + 7X_8 + 42X_9,\\
p_2(\mathbf{X}) &= 28X_2 + 42X_3 + 42X_6 + 42X_7,\\
p_3(\mathbf{X}) &= 28X_3 + 42X_6 + 42X_7,\\
p_4(\mathbf{X}) &= 28X_6 + 35X_7,
\end{align*}  By a slight abuse of notation, we let
\begin{align}
I(\alpha) := I\left(\mathbf{s}_{\alpha}\right) = &\big( p_1(\mathbf{s}_{\alpha}), p_2(\mathbf{s}_{\alpha}), p_3(\mathbf{s}_{\alpha}), p_4(\mathbf{s}_{\alpha}) \big) \le\left(\mathbb{Z}/49\mathbb{Z}\right)[\mathbf{s}_{\alpha}].\label{congruenceidealdefn}
\end{align}  

\begin{definition}
Let $f_{\alpha}\in\mathcal{V}^{(1)}_n$ as in (\ref{falpha}) and $n\equiv 3\pmod{7}$, and $\mathbf{s}_{\alpha}$ defined as in (\ref{salphavector}).  The ideal $I(\alpha)$ is the \textit{congruence ideal} associated with $f_{\alpha}$.
\end{definition}

\begin{lemma}\label{dkdddisshsi1}
Let $n\equiv 3\pmod{7}$, $f_{\alpha}\in\mathcal{V}^{(1)}_n$ as in (\ref{falpha}), and let $I(\alpha)$ be the associated congruence ideal.  If $I(\alpha)=(0)$, then
\begin{align*}
\frac{1}{7}U^{(1)}\left( f_{\alpha} \right)\in\mathcal{V}^{(0)}_{7n}.
\end{align*}
\end{lemma}

\begin{proof}
Of course, (\ref{idealel1}) is the very first relation of $I(\alpha)$.  The relations in Tables \ref{tablew0}, \ref{tablew1} are taken mod 7.  We therefore multiply each of these relations by 7, and show that the results are members of $I(\alpha)$.  If we consider the first relation of Table \ref{tablew0}, for example, we get
\begin{align*}
86(7)\left( 3s_0(1)+2s_0(2)+2s_0(3)+3s_1(0)+4s_1(1) \right) =& 42p_1(\mathbf{s}_{\alpha})-29p_2(\mathbf{s}_{\alpha})+64p_3(\mathbf{s}_{\alpha})-1764 s_0(4)\\ &- 294 s_0(5) - 196 s_1(1) - 294 s_1(2) - 1764 s_1(3)\\
\equiv& 42p_1(\mathbf{s}_{\alpha})-29p_2(\mathbf{s}_{\alpha})+64p_3(\mathbf{s}_{\alpha}) \pmod{49}.
\end{align*}  Notice that 86 is coprime with 7.  Therefore, 

\begin{align*}
7\left( 3s_0(1)+2s_0(2)+2s_0(3)+3s_1(0)+4s_1(1) \right)\in I(\alpha),
\end{align*} and will vanish if $I(\alpha)=(0).$  We verify membership of each 7-multiple of the elements of Tables \ref{tablew0}, \ref{tablew1} in our Mathematica supplement.
\end{proof}

We see, then, that the congruence ideal contains the information related to the complex interactions between the basis functions that make up $f_{\alpha}$.  On applying $U^{(1)}$ to a typical element of $\mathcal{V}^{(1)}_n$, as we have seen, we do not necessarily gain a 7-multiple of an element in $\mathcal{V}^{(0)}_{7n}$.  However, if the coefficients of $f_{\alpha}$ are so chosen that $I(\alpha)=(0)$, then we \textit{will} gain such a function.  In particular, as we checked in (\ref{checkingL1forfirst}) and our Mathematica supplement, we have
\begin{align}
I(1)=(0).\label{L1congidealis0}
\end{align}
But, as we already noted, this is not enough.  What we need to determine is whether the interactions encoded in $I(\alpha)$ are also present in the successor $f_{\alpha+1}$ to $f_{\alpha}$.  We can of course construct the ideal $I(\alpha+1)$ analogous to $I(\alpha)$, and we want to determine the relationship that the former bears with respect to the latter.  If we can prove that $I(\alpha)=(0)$ implies $I(\alpha+1)=(0)$, then the structure of the coefficients of $f_{\alpha}$ which ensures that a power of 7 will be gained after applying $U^{(0)}\circ U^{(1)}$ will also be present with $f_{\alpha+1}$.

\begin{lemma}\label{stablethmlem}
Let $n\equiv 3\pmod{7}$, $f_{\alpha}\in\mathcal{V}^{(1)}_n$ as in (\ref{falpha}) with congruence ideal $I(\alpha)$, and let
\begin{align*}
f_{\alpha+1} =& \frac{1}{7} U^{(0)}\circ U^{(1)}\left( f_{\alpha} \right) = \frac{1}{(1+7x)^{49n+3}}\sum_{\substack{0\le\delta\le 1\\ w\ge 1-\delta}}s_{\alpha+1,\delta}(m)7^{\theta^{(1)}_{\delta}(w)}y^{\delta}x^w.
\end{align*} Define $\mathbf{s}_{\alpha+1}$ in a manner analogous to that of $\mathbf{s}_{\alpha}$, with congruence ideal $I(\alpha+1)$.  Then we have
\begin{align*}
I(\alpha+1)\subseteq I(\alpha).
\end{align*}
\end{lemma}

This critical lemma allows us to complete the proof of Theorem \ref{Main}, and with it Theorem \ref{Thm12}.

\begin{proof}

Recalling (\ref{tgammadefn}), adjusted with the notation of (\ref{falpha}), we have

\begin{align*}
t_{\gamma}(r) = \sum_{\substack{0\le\beta\le 1\\ m\ge 1-\beta}} s_{\alpha,\beta}(m)h^{(1)}_{\beta\gamma}(m,n,r)7^{\pi^{(1)}_{\beta\gamma}(m,r)+\theta^{(1)}_{\beta}(m) - \theta^{(0)}_{\gamma}(r) -1}.
\end{align*}  We take (\ref{u23fsad}), replacing $f$ with $U^{(1)}(f_{\alpha})$.  By Theorem \ref{goingv0tov1}, we have
\begin{align*}
s_{\alpha+1,\delta}(w) =& \sum_{\substack{0\le\gamma\le 1\\ r\ge 1-\gamma}}\left(\sum_{\substack{0\le\beta\le 1\\ m\ge 1-\beta}} s_{\alpha,\beta}(m)h^{(1)}_{\beta\gamma}(m,n,r)7^{\pi^{(1)}_{\beta\gamma}(m,r)+\theta^{(1)}_{\beta}(m) - \theta^{(0)}_{\gamma}(r) -1}\right)\\ &\times h^{(0)}_{\gamma\delta}(r,7n,w) 7^{\pi^{(0)}_{\gamma\delta}(r,w) + \theta^{(0)}_{\gamma}(r) -\theta^{(1)}_{\delta}(w)}\\
=& \sum_{\substack{0\le\gamma\le 1\\ r\ge 1-\gamma}}t_{\gamma}(r) h^{(0)}_{\gamma\delta}(r,7n,w) 7^{\pi^{(0)}_{\gamma\delta}(r,w) + \theta^{(0)}_{\gamma}(r) -\theta^{(1)}_{\delta}(w)}.
\end{align*}  Notice that for $t_{\gamma}(r)\in\mathbb{Z}$ we must have $s_{\alpha+1,\delta}(w)\in\mathbb{Z}$.  

We want to examine, say, $p_1(\mathbf{s}_{\alpha+1})$ modulo 49.  Notice that $0\le m\le 5$ for all of our deviant cases; for $m\ge 6$, we have 
\begin{align*}
\pi^{(1)}_{\beta\gamma}(m,r)+\theta^{(1)}_{\beta}(m) - \theta^{(0)}_{\gamma}(r)\ge 1.
\end{align*}   Moreover, for $m\ge 7$, we have

\begin{align*}
\pi^{(1)}_{\beta\gamma}(m,r)+\theta^{(1)}_{\beta}(m) - \theta^{(0)}_{\gamma}(r)\ge 2.
\end{align*}  Similarly, for $r\ge 7$, we have
\begin{align*}
\pi^{(0)}_{\gamma\delta}(r,w)+\theta^{(0)}_{\gamma}(r) - \theta^{(1)}_{\delta}(w)\ge 2.
\end{align*}  As such, to properly study $p_1(\mathbf{s}_{\alpha+1})$ modulo 49, we need only take $0\le m\le 6$, $0\le r\le 6$.

We compute this explicitly in our Mathematica supplement.  Notably, by using Theorem \ref{hcongred37}, we can reduce our auxiliary function $h^{(0)}_{\beta\gamma}(r,7n,w)$ to $h^{(0)}_{\beta\gamma}(r,7,w)$.  Doing so gives us the following congruence:
\begin{align*}
p_1(\mathbf{s}_{\alpha+1}) \equiv& 25 h^{(1)}_{00}(1,n,1) s_{\alpha,0}(1) + 7 h^{(1)}_{00}(1,n,2) s_{\alpha,0}(1) + 35 h^{(1)}_{01}(1,n,0) s_{\alpha,0}(1) + 25 h^{(1)}_{00}(2,n,1) s_{\alpha,0}(2)\\
 &+ 7 h^{(1)}_{00}(2,n,2) s_{\alpha,0}(2) + 35 h^{(1)}_{01}(2,n,0) s_{\alpha,0}(2) + 
 25 h^{(1)}_{00}(3,n,1) s_{\alpha,0}(3) + 7 h^{(1)}_{00}(3,n,2) s_{\alpha,0}(3)\\ 
&+ 35 h^{(1)}_{01}(3,n,0) s_{\alpha,0}(3) + 28 h^{(1)}_{00}(4,n,1) s_{\alpha,0}(4) + 
 28 h^{(1)}_{00}(5,n,1) s_{\alpha,0}(5) + 25 h^{(1)}_{10}(0,n,1) s_{\alpha,1}(0)\\
&+ 7 h^{(1)}_{10}(0,n,2) s_{\alpha,1}(0) + 35 h^{(1)}_{11}(0,n,0) s_{\alpha,1}(0) + 
 25 h^{(1)}_{10}(1,n,1) s_{\alpha,1}(1) + 7 h^{(1)}_{10}(1,n,2) s_{\alpha,1}(1)\\
&+ 35 h^{(1)}_{11}(1,n,0) s_{\alpha,1}(1) + 28 h^{(1)}_{10}(2,n,1) s_{\alpha,1}(2) + 
 28 h^{(1)}_{10}(3,n,1) s_{\alpha,1}(3)\pmod{49}.
\end{align*}  Notice that the remaining non-vanishing coefficients are also subject to Theorem \ref{hcongred37}.  We therefore reduce $h^{(1)}_{\beta\gamma}(m,n,r)$ to $h^{(1)}_{\beta\gamma}(m,3,r)$ mod 49, and we have
\begin{align*}
p_1(\mathbf{s}_{\alpha+1}) \equiv& 18 s_{\alpha,0}(1) + 38 s_{\alpha,0}(2) + 32 s_{\alpha,0}(3) + 21 s_{\alpha,0}(4) + 28 s_{\alpha,0}(5) + 4 s_{\alpha,1}(0) + 45 s_{\alpha,1}(1) + 28 s_{\alpha,1}(2) + 21 s_{\alpha,1}(3)\\
\equiv& 46\left( 43 s_{\alpha,0}(1) + 20 s_{\alpha,0}(2) + 22 s_{\alpha,0}(3) + 42 s_{\alpha,0}(4) + 7 s_{\alpha,0}(5) + 15 s_{\alpha,1}(0) + 34 s_{\alpha,1}(1) + 7 s_{\alpha,1}(2) + 42 s_{\alpha,1}(3) \right)\\
\equiv&46\cdot\left( p_1(\mathbf{s}_{\alpha}) - p_2(\mathbf{s}_{\alpha}) \right)\pmod{49}.
\end{align*}  That is,
\begin{align*}
p_1(\mathbf{s}_{\alpha+1})&\in I(\alpha).
\end{align*}  We similarly determine in our Mathematica supplement that
\begin{align*}
p_2(\mathbf{s}_{\alpha+1}), p_3(\mathbf{s}_{\alpha+1}), p_4(\mathbf{s}_{\alpha+1})&\in I(\alpha).
\end{align*}  We therefore have
\begin{align*}
(0)\subseteq I(\alpha+1)\subseteq I(\alpha).
\end{align*}
\end{proof}

This gives us enough to prove Theorem \ref{Main}.

\subsection{Proof of Theorem \ref{Main} (I): Demonstrating Ideal Stability}

Suppose that for some $\alpha\ge 1$ we have
\begin{align*}
\frac{1}{7^{\alpha-1}}L_{2\alpha-1} = f_{\alpha}\in\mathcal{V}^{(1)}_{n},
\end{align*} with $n\equiv 3\pmod{7}$ and $f_{\alpha}$ defined as in (\ref{falpha}).  In particular, by hypothesis,
\begin{align*}
L_{2\alpha-1}\equiv 0\pmod{7^{\alpha-1}}.
\end{align*}  We define the associated congruence ideal $I(\alpha)$ as in (\ref{congruenceidealdefn}).  Suppose that
\begin{align}
I(\alpha)=(0).\label{inproofieq0}
\end{align}  Then by Lemma \ref{dkdddisshsi1} we have 
\begin{align*}
\frac{1}{7^{\alpha}}L_{2\alpha} = \frac{1}{7^{\alpha}}U^{(1)}\left( L_{2\alpha-1} \right) = \frac{1}{7}U^{(1)}\left( f_{\alpha} \right)\in\mathcal{V}^{(0)}_{7n}.
\end{align*}  Thus, $L_{2\alpha}$ is divisible by $7^{\alpha}$.  Moreover, by Lemma \ref{stablethmlem},
\begin{align*}
\frac{1}{7^{\alpha}}L_{2\alpha+1} = \frac{1}{7^{\alpha}}U^{(0)}\circ U^{(1)}\left( L_{2\alpha-1} \right) = \frac{1}{7}U^{(0)}\circ U^{(1)}\left( f_{\alpha} \right) = f_{\alpha+1}\in\mathcal{V}^{(1)}_{49n+3},
\end{align*} and by (\ref{inproofieq0})
\begin{align*}
(0)\subseteq I(\alpha+1)\subseteq I(\alpha)=(0),
\end{align*} whence $I(\alpha+1)=(0)$.  To complete the induction, we note from (\ref{L1expressxy}) that $L_{1}\in\mathcal{V}^{(1)}_{3}$, and by (\ref{L1congidealis0}) we have $I(1)=(0)$.

\subsection{Proof of Theorem \ref{Main} (II): Computing the Localizing Factor}

Checking the values of $\psi(\alpha)$ is a straightforward exercise in elementary number theory.  If for some $\alpha\ge 1$ we have
\begin{align*}
\psi(2\alpha-1) = \left\lfloor \frac{7^{2\alpha}}{16} \right\rfloor,
\end{align*} as in (\ref{psidefn}), then we know from Theorem \ref{thmrelAA} that on applying $U^{(1)}$ to $L_{2\alpha-1}$, our resultant denominator will be $1+7x$ raised to the power of
\begin{align*}
\psi(2\alpha) = 7\psi(2\alpha-1) &= 7\left\lfloor \frac{7^{2\alpha}}{16} \right\rfloor = 7\left( \frac{7^{2\alpha}-1}{16} \right) = \frac{7^{2\alpha+1}-7}{16} = \left\lfloor \frac{7^{2\alpha+1}}{16} \right\rfloor.
\end{align*}  We may similarly verify that
\begin{align*}
\psi(2\alpha+1) = 7\psi(2\alpha)+3 = \left\lfloor \frac{7^{2\alpha+2}}{16} \right\rfloor.
\end{align*}  We note that
\begin{align*}
\psi(1) = \left\lfloor \frac{7^{2}}{16} \right\rfloor = 3,
\end{align*} which matches the localizing factor we have for $L_1$ in (\ref{L1expressxy}).

\subsection{Proof of Theorem \ref{Main} (III): Accounting for the Anomalous Term $r_L$}

Finally, we note the occurence of the term $r_L$.  Notice that for a given element of $\mathcal{V}^{(1)}$, the power of 7 attached to $x^m$ for $1\le m\le 3$ is $-1$.  This is also true for $y, yx$.  If we take 
\begin{align*}
&s_0(1)x+s_0(2)x^2+s_0(3)x^3+s_1(0)y+s_1(1)xy\\
&\equiv s_0(1)\left(x+s_0(2)s_0(1)^{-1}x^2+s_0(3)s_0(1)^{-1}x^3+s_1(0)s_0(1)^{-1}y+s_1(1)s_0(1)^{-1}xy \right)\pmod{7},
\end{align*} we should expect
\begin{align*}
s_0(2)&\equiv 3s_0(1)\pmod{7},\\
s_0(3)&\equiv s_0(1)\pmod{7},\\
s_1(0)&\equiv 6s_0(1)\pmod{7},\\
s_1(1)&\equiv 5s_0(1)\pmod{7};
\end{align*} or, equivalently,
\begin{align*}
7(s_0(2)- 3s_0(1))&\equiv 0\pmod{49},\\
7(s_0(3)- s_0(1))&\equiv 0\pmod{49},\\
7(s_1(0)- 6s_0(1))&\equiv 0\pmod{49},\\
7(s_1(1)- 5s_0(1))&\equiv 0\pmod{49}.
\end{align*}  Indeed, as we verify in our Mathematica supplement, we have
\begin{align*}
7(s_0(2) - 3s_0(1)), 7(s_0(3) - s_0(1)), 7(s_1(0) - 6s_0(1)), 7(s_1(1) - 5s_1(1))\in I(\alpha).
\end{align*}  Therefore, for $I(\alpha)=(0)$, we have
\begin{align*}
\frac{1}{7}(s_0(1)x+s_0(2)x^2+s_0(3)x^3+s_1(0)y+s_1(1)xy) = s_0(1)r_L + M_0(x)+yM_1(x),
\end{align*} for some $M_0(x), M_1(x)\in\mathbb{Z}[x].$ \begin{flushright}\qedsymbol\end{flushright}

\subsection{Proof of Theorem \ref{Thm12}}

Recall that by Theorem \ref{Main}, we have
\begin{align*}
\frac{(1+7x)^{\psi}}{7^{\beta}}\cdot L_{\alpha} + k_{\alpha}\cdot r_L \in\mathbb{Z}[x]\oplus y\mathbb{Z}[x],
\end{align*} in which $k_{\alpha}\in\mathbb{Z}$ for all $\alpha\ge 1$.  Moreover, by (\ref{rLisint}), $r_L$ is a power series with integral coefficients, thus implying that $\frac{(1+7x)^{\psi}}{7^{\beta}}\cdot L_{\alpha}$ is an integral power series.  Since $1+7x$ is not divisible by 7, we must have
\begin{align*}
L_{\alpha}\equiv 0\pmod {7^{\beta}}.
\end{align*}
\begin{flushright}\qedsymbol\end{flushright}

\section{Further Work}\label{sectionfurther}

Despite the tediousness of the proof, the underlying elegance, especially of the last step, is astonishing.  We see that each generating function $L_{\alpha}$ is composed of a combination of basis functions $y^{\beta}x^m$ in just the right way so that any deviant terms will disappear in passing to its successor... and that this composition is itself very precisely inherited by its successor.  We do not yet have a comprehensive understanding of these deviant terms; nevertheless, they may play a substantial role in determining why congruence families exist at all, especially over curves of nontrivial topology.

This sort of cancellation property was not necessary in classical congruence families, e.g., those of Ramanujan \cite{Ramanujan}, but we believe that it may become a more important component of future proofs, especially of the more difficult families.

We finish now by considering the following additional questions which are raised by this proof:

\subsection{Congruence Families of Higher Genus and Cusp Count}

Our congruence family from Theorem \ref{Thm12} was associated with a modular curve of genus 1 and cusp count 4.  The natural question that arises from our work is whether our methods can be used to prove a given congruence family when either the genus or the cusp count of the associated curve $\mathrm{X}$ is increased.

It is important to note that in the case of the main result of this paper, the genus is a complicating factor, but it does not appear to fundamentally alter the accessibility of the proof.  Our methods are in large measure identical to those used in previous applications of localization to prove congruences, except for a greater degree of necessary calculations, together with a bit of ``index gymnastics."  As such, any congruence families in which $\mathrm{X}$ has higher genus ought to be accessible to proof, provided that the cusp count of $\mathrm{X}$ remains 4 or less.  We are currently investigating possible congruence families whose curves carry such a topology.

A more ambitious problem is for us to adapt our methods to encompass congruence families in which $\mathrm{X}$ has a cusp count greater than 4.  These are among the most difficult and the least-understood congruence families.  The proofs of these families---when we have them at all---are not generally easy to understand, and often involve algebraic structures or functions which are not clearly understood.  

As an example, we consider the Andrews--Sellers congruence family.  Its associated modular curve, $\mathrm{X}_0(20)$, has genus 1 and cusp count 6.  The family was proposed in 1994, and concerned the generalized 2-color Frobenius function $c\phi_2(n)$.  It resisted proof until Paule and Radu's work in 2012.  The techniques developed by Paule and Radu have been extremely important and applicable.  On the other hand, their proof involves certain algebraic structures, e.g., a pair of rank 2 $\mathbb{Z}[X]$ modules, which do not arise from a clear theoretical understanding.

The situation is analogous to the proofs of the irrationality of $\sqrt{2}$.  The classical ``Pythagorean" proof is especially elegant, but it does not highlight \textit{why} $\sqrt{2}$ is irrational in the same way as a proof which utilizes the property of unique factorization in $\mathbb{Z}$.  The latter proof relates irrationality to a much deeper and more general property of $\mathbb{Z}$, and immediately allows us to generalize to a much broader class of irrational numbers.

The appeal of localization is that our proof not only provides an algebraic framework which arises naturally from the topology of the underlying Riemann surface (and which easily reduces to the classical case for simple topologies), but it also allows us to highlight what makes some congruences more difficult to prove than others.  The ideal structure that plays such a prominent role in our proof in this article emphasizes the complex interaction and inheritance between basis functions which might not be noticeable otherwise.

\subsection{Ideal Vs. Kernel Approaches}

This article constitutes the second collaboration between Banerjee and Smoot to prove a congruence family using the localization method.  Our first result \cite{Banerjee} concerned $d_5(n)$ modulo powers of 5.  In that paper we used an approach which is equivalent but which relies on the properties of certain vector spaces.  Instead of working with a sequence of ideals, we build a vector space, and work with a linear operator $\Omega$ on that vector space.  For every odd-indexed $L_{2\alpha-1}$ we have a certain vector $v_{\alpha}$.  Proving the congruence family amounts to proving that 
\begin{align*}
v_{\alpha}\in\mathrm{ker}\left( \Omega \right)\text{ implies } v_{\alpha+1}\in\mathrm{ker}\left( \Omega \right).
\end{align*}  The approach is equivalent to our approach in this article, except that in our current case we would be forced to work over a $\mathbb{Z}/49\mathbb{Z}$-module, rather than a vector space.

Part of the reason that we used a different formulation is that we recognize limitations to the current localization method.  It cannot currently be used to prove congruences associated with a modular curve of cusp count 6.  We hope to modify the method to properly account for the most difficult congruence families.

With this goal in mind, our two papers show that one can consider modifications to the method either from the perspective of linear operators on certain modules and vector spaces, or from the perspective of ideal chains associated with a given polynomial ring.  One of these approaches may ultimately prove more useful to the task than the other.

\pagebreak
\section{Appendix}

\begin{align*}
\theta^{(1)}_{0}(m) &=\begin{cases} -1, & 1\le m\le 3,\\ \left\lfloor \frac{7m-28}{9} \right\rfloor & m\ge 4\end{cases}\\
\theta^{(1)}_{1}(m) &=\begin{cases} -1, & 0\le m\le 1,\\ \left\lfloor \frac{7m-14}{9} \right\rfloor & m\ge 2\end{cases}\\
\theta^{(0)}_{0}(m) &=\begin{cases} 0, & 1\le m\le 2,\\ \left\lfloor \frac{7m-16}{9} \right\rfloor & m\ge 3\end{cases}\\
\theta^{(0)}_{1}(m) &=\begin{cases} 0 & 0\le m\le 1,\\ \left\lfloor \frac{7m-5}{9} \right\rfloor & m\ge 2\end{cases}
\end{align*}

\begin{align*}
\pi^{(1)}_{00}(m,r) &=\begin{cases}
\left\lfloor \frac{7r+2}{9} \right\rfloor & 1\le m\le 3\text{ and } r\ge 9,\\
\max\left(0,\left\lfloor \frac{7r-m}{9} \right\rfloor\right)& \text{ otherwise}\end{cases}\\
\pi^{(1)}_{01}(m,r) &=\max\left(0,\left\lfloor \frac{7r-m+16}{9} \right\rfloor\right)\\
\pi^{(1)}_{10}(m,r) &=\begin{cases}
\left\lfloor \frac{7r+2}{9} \right\rfloor & 0\le m\le 1\text{ and } r\ge 11,\\
\left\lfloor \frac{7r-3}{9} \right\rfloor & m=0\text{ and } 1\le r\le 10,\\
\max\left(0,\left\lfloor \frac{7r-m-2}{9} \right\rfloor\right)& \text{ otherwise}
\end{cases}\\
\pi^{(1)}_{11}(m,r) &=\begin{cases}
\left\lfloor \frac{7r+4}{9} \right\rfloor + 1 & 0\le m\le 1\text{ and } r\ge 7,\\
\left\lfloor \frac{7r+2}{9} \right\rfloor + 1 & m=0,\\
\max\left(0,\left\lfloor \frac{7r-m+3}{9} \right\rfloor + 1\right)& \text{ otherwise}
\end{cases}\\
\pi^{(0)}_{00}(m,r) &=\begin{cases} -1 & 1\le r\le 3,\\ \max\left(0,\left\lfloor \frac{7r-m+4}{9} \right\rfloor - 3\right) & r\ge 4 \end{cases}\\
\pi^{(0)}_{01}(m,r) &=\begin{cases} -1 & 0\le r\le 1,\\ \max\left(0,\left\lfloor \frac{7r-m}{9} \right\rfloor - 1\right) & r\ge 4 \end{cases}\\
\pi^{(0)}_{10}(m,r) &=\begin{cases}
 -1 & 1\le r\le 3,\\ 
\max\left(0,\left\lfloor \frac{7r-1}{9} \right\rfloor - 3\right) & m=0, r\ge 4 \\
\max\left(0,\left\lfloor \frac{7r-m}{9} \right\rfloor - 3\right) & r\ge 4
 \end{cases}\\
\pi^{(0)}_{11}(m,r) &=\begin{cases}
 -1 & 0\le r\le 1,\\
\max\left(0,\left\lfloor \frac{7r-m+4}{9} \right\rfloor-2\right) & m=0, r\ge 2\\
\max\left(0,\left\lfloor \frac{7r-m+5}{9} \right\rfloor-2\right) & r\ge 2
\end{cases}
\end{align*}

\begin{align*}
b_0(z) &= z^{14}\\
b_1(z) 
&= -2 z^7 - 35 z^8 - 84 z^9 - 133 z^{10} + 161 z^{11} + 112 z^{12} + 63 z^{13} - 96 z^{14}\\
b_2(z) 
&= 1 - 14 z + 35 z^2 + 84 z^3 + 476 z^4 - 1876 z^5 - 4228 z^6 - 8859 z^7 + 9940 z^8\\ 
&+ 9254 z^9 + 5138 z^{10} + 1022 z^{11} - 5152 z^{12} - 5152 z^{13} + 4224 z^{14}\\
b_3(z)
&= -12 + 203 z - 826 z^2 - 826 z^3 + 25242 z^4 - 111958 z^5 - 308154 z^6\\ 
&- 1027606 z^7 + 944559 z^8 + 345632 z^9 + 203924 z^{10} - 218015 z^{11} + 67312 z^{12}\\ 
&+ 190400 z^{13} - 112640 z^{14}\\
b_4(z)
&= 66 - 1260 z + 7364 z^2 - 6650 z^3 + 136430 z^4 - 345436 z^5 - 4566688 z^6\\
&- 38513350 z^7 + 41646514 z^8 + 4955069 z^9 - 5994226 z^{10} + 4396224 z^{11}\\
&+ 683200 z^{12} - 4193280 z^{13} + 2027520 z^{14}\\
b_5(z) 
&= -220 + 4515 z - 32480 z^2 + 99967 z^3 - 24493 z^4 + 7498575 z^5 - 5736031 z^6\\ 
&- 565652279 z^7 + 610835141 z^8 - 53989607 z^9 + 34022968 z^{10} - 30421440 z^{11}\\ 
&- 33660928 z^{12} + 61071360 z^{13} - 25952256 z^{14}\\
b_6(z) 
&= 495 - 10500 z + 82306 z^2 - 354480 z^3 + 669914 z^4 + 23325603 z^5 + 256881359 z^6\\
&- 3911740166 z^7 + 4057736697 z^8 - 673766856 z^9 + 182358400 z^{10} - 54240256 z^{11}\\ 
&+ 501760000 z^{12} - 616562688 z^{13} + 242221056 z^{14}\\
b_7(z)
&= -792 + 16758 z - 127988 z^2 + 493185 z^3 - 3807692 z^4 + 41690572 z^5 + 1462044654 z^6\\ 
&- 14498962786 z^7 + 11696357232 z^8 + 2668196608 z^9 - 1949538304 z^{10}\\ 
&+ 2020085760 z^{11} - 4193910784 z^{12} + 4393009152 z^{13} - 1660944384 z^{14}\\
b_8(z)
&=924 - 18816 z + 122500 z^2 - 105938 z^3 + 2849350 z^4 - 84220857 z^5 + 4057736697 z^6\\ 
&- 31293921328 z^7 + 16440406976 z^8 + 11942708736 z^9 + 2743967744 z^{10}\\ 
&- 11615600640 z^{11} + 21576024064 z^{12} - 22020096000 z^{13} + 8304721920 z^{14} \\
b_9(z)
&=-792 + 14910 z - 65744 z^2 - 475335 z^3 + 4252871 z^4 - 53989607 z^5 + 4886681128 z^6\\ 
&- 36201745856 z^7 - 2936847872 z^8 + 30714163200 z^9 - 802586624 z^{10} + 26205749248 z^{11}\\ 
&- 68115496960 z^{12} + 75749130240 z^{13} - 29527900160 z^{14}\\
b_{10}(z)
&=495 - 8190 z + 10675 z^2 + 549528 z^3 - 5994226 z^4 + 39640552 z^5 + 2665376896 z^6\\ 
&- 19718835200 z^7 - 18705154048 z^8 - 11319246848 z^9 + 35764305920 z^{10}\\ 
&- 13946060800 z^{11} + 123547418624 z^{12} - 169114337280 z^{13} + 70866960384 z^{14}\\
b_{11}(z)
&=-220 + 2975 z + 8414 z^2 - 218015 z^3 + 1631392 z^4 + 22120448 z^5 + 483614208 z^6\\ 
&- 4209074176 z^7 - 10097590272 z^8 - 29349117952 z^9 + 52936310784 z^{10} - 13857980416 z^{11}\\ 
&- 110863843328 z^{12} + 217969590272 z^{13} - 103079215104 z^{14}\\
b_{12}(z)
&=66 - 644 z - 5152 z^2 + 8176 z^3 + 328832 z^4 + 4738048 z^5 + 40714240 z^6 - 290291712 z^7\\
&- 1108344832 z^8 - 3934257152 z^9 + 7985954816 z^{10} + 11274289152 z^{11} + 37580963840 z^{12}\\ 
&- 120259084288 z^{13} + 68719476736 z^{14}\\
b_{13}(z)
&=-12 + 63 z + 896 z^2 + 10304 z^3 - 68096 z^4 - 344064 z^5 - 1146880 z^6 - 524288 z^7\\
b_{14}(z) &= 1.
\end{align*}

\begin{align*}
a_0(z)
&= 3 x + 175 x^2 + 4606 x^3 + 73059 x^4 + 794731 x^5 + 6487502 x^6 + 42824236 x^7 + 238945119 x^8\\
&+ 1117547851 x^9 + 4162186322 x^{10} + 11541131602 x^{11} + 22033069422 x^{12} + 25705247659 x^{13}\\
&+ 13841287201 x^{14}\\
a_1(z) 
&= 203 x + 12005 x^2 + 322077 x^3 + 5251673 x^4 + 59434354 x^5 + 511403396 x^6 + 3581353209 x^7\\
&+ 21042935438 x^8 + 101911799164 x^9 + 387181329563 x^{10} + 1085915564370 x^{11}\\
&+ 2089469416853 x^{12} + 2453862488063 x^{13} + 1328763571296 x^{14}\\
a_2(z) 
&= 6265 x + 376201 x^2 + 10314990 x^3 + 173583039 x^4 + 2053551290 x^5 + 18689384000 x^6\\
&+ 138788289969 x^7 + 855550469291 x^8 + 4274768767815 x^9 + 16532786315885 x^{10}\\
&+ 46861231432855 x^{11} + 90848559682384 x^{12} + 107376751451872 x^{13} + 58465597137024 x^{14}\\
a_3(z)
&= 116459 x + 7115969 x^2 + 200101839 x^3 + 3492479508 x^4 + 43419223008 x^5 + 419167823718 x^6\\
&+ 3295778614744 x^7 + 21226671646068 x^8 + 109033294196141 x^9 + 428467687129601 x^{10}\\
&+ 1226396760506785 x^{11} + 2394725729734352 x^{12} + 2847983254477760 x^{13} + 1559082590320640 x^{14}\\
a_4(z)
&= 1449833 x + 90415745 x^2 + 2619811999 x^3 + 47714242940 x^4 + 626857433013 x^5\\ 
&+ 6431855317139 x^6 + 53404125506473 x^7 + 357773867642002 x^8 + 1882980021363602 x^9\\ 
&+ 7505614687422502 x^{10} + 21678572346541824 x^{11} + 42622306155271360 x^{12}\\ 
&+ 50994149398993920 x^{13} + 28063486625771520 x^{14}\\
a_5(z) 
&= 12707135 x + 812295414 x^2 + 24406700227 x^3 + 467289947852 x^4 + 6523210723638 x^5\\ 
&+ 71180537383858 x^6 + 621830080144679 x^7 + 4313683223069388 x^8 + 23190084180310873 x^9\\ 
&+ 93616781749765064 x^{10} + 272672550787560000 x^{11} + 539628856248745984 x^{12}\\
&+ 649372072346640384 x^{13} + 359212628809875456 x^{14}\\
a_6(z) 
&= 80168088 x + 5286545768 x^2 + 166136541193 x^3 + 3372698052990 x^4 + 50267896083671 x^5\\ 
&+ 582755302614765 x^6 + 5332050237612588 x^7 + 38130306660737921 x^8 + 208794282026692488 x^9\\
&+ 852473027477812608 x^{10} + 2502347102814884864 x^{11} + 4983330104745803776 x^{12}\\ 
&+ 6030418071535484928 x^{13} + 3352651202225504256 x^{14}\\
a_7(z)
&= -3 + 365390731 x + 25095939659 x^2 + 834433504065 x^3 + 18132408865284 x^4\\ 
&+ 289513052752098 x^5 + 3555914358207064 x^6 + 33901256261028254 x^7 + 248847657182595856 x^8\\
&+ 1384458034838074112 x^9 + 5709778330183280640 x^{10} + 16881992402381180928 x^{11}\\
&+ 33821264829712826368 x^{12} + 41149283911545126912 x^{13} + 22989608243832029184 x^{14}\\
a_8(z)
&=-119 + 1187495288 x + 86220830934 x^2 + 3079413154580 x^3 + 72315867551242 x^4\\ 
&+ 1238343447960010 x^5 + 16042203009007441 x^6 + 158492327221426800 x^7\\ 
&+ 1189441617391337408 x^8 + 6708389085147929088 x^9 + 27916706629722251264 x^{10}\\ 
&+ 83096997943085629440 x^{11} + 167427147653571936256 x^{12} + 204766748751880519680 x^{13}\\
&+ 114948041219160145920 x^{14}
\end{align*}
\begin{align*}
a_9(z)
&=-1988 + 2661685929 x + 209205847900 x^2 + 8181772833318 x^3 + 209365036375208 x^4\\
&+ 3838510233679961 x^5 + 52139774281593432 x^6 + 530898816397465152 x^7\\
&+ 4058869323511299584 x^8 + 23161327540107079680 x^9 + 97163280575812763648 x^{10}\\
&+ 291030344013883113472 x^{11} + 589573740171907563520 x^{12} + 724686857267692175360 x^{13}\\
&+ 408704146557013852160 x^{14}\\
a_{10}(z)
&=-17913 + 3854368588 x + 341084691626 x^2 + 14946334588157 x^3 + 418884673555446 x^4\\
&+ 8183050172059224 x^5 + 115757691604548224 x^6 + 1208423052153490432 x^7\\
&+ 9382000579851124736 x^8 + 54076398124284641280 x^9 + 228495346771705724928 x^{10}\\
&+ 688405122921636298752 x^{11} + 1401821100823034200064 x^{12} + 1731419384505167773696 x^{13}\\ 
&+ 980889951736833245184 x^{14}\\
a_{11}(z)
&=-91840 + 3161053504 x + 337306540599 x^2 + 16974499008929 x^3 + 520991685008480 x^4\\
&+ 10762401101720576 x^5 + 157420717177934336 x^6 + 1676798937863856128 x^7\\
&+ 13184178950079578112 x^8 + 76647275568092938240 x^9 + 325973004513076838400 x^{10}\\
&+ 987439454642002460672 x^{11} + 2020696994730586669056 x^{12} + 2507426478357575892992 x^{13}\\
&+ 1426749020708121083904 x^{14}\\
a_{12}(z)
&=-254016 + 1050345289 x + 154541941232 x^2 + 9126241254480 x^3 + 304769343501952 x^4\\
&+ 6593965302027264 x^5 + 99034088979775488 x^6 + 1071829647175024640 x^7\\
&+ 8515062058317774848 x^8 + 49869621987073064960 x^9 + 213334790308255236096 x^{10}\\
&+ 649528704899328507904 x^{11} + 1335454068490190716928 x^{12} + 1664540524159474597888 x^{13}\\
&+ 951166013805414055936 x^{14}\\
a_{13}(z)
&=-296009 - 12511233 x - 224214592 x^2 - 2204529600 x^3 - 12808201728 x^4 - 43783176192 x^5\\
&- 80957571072 x^6 - 61681958912 x^7\\
a_{14}(z) &= 1.
\end{align*}

\section{Acknowledgments}
We would like to thank the anonymous referee for their insightful comments and remarks which made a substantial improvement in the exposition of the paper.

The first author was funded by the Austrian Science Fund (FWF): W1214-N15, Project DK6.  

The second author was funded in whole by the Austrian Science Fund (FWF) Einzelprojekte 10.55776/P33933, ``Partition Congruences by the Localization Method."  For open access purposes, the authors have applied a CC BY public copyright license to any author-accepted manuscript version arising from this submission.  The second author is the Principal Investigator of the aforementioned FWF Project.

We would like to thank the Austrian Government and People for their generous support.

\end{document}